\documentclass[12pt,psamsfonts]{amsart}
\usepackage{graphicx}
\usepackage{array}
\usepackage{color}
\usepackage{epsfig}
\usepackage{subfig}
\usepackage{url}
\usepackage{ulem}
\newcolumntype{L}{>{\displaystyle}l}
\newcolumntype{C}{>{\displaystyle}c}
\newcolumntype{R}{>{\displaystyle}r}

\newcommand{\R}{\ensuremath{\mathbb{R}}}

\newcommand{\Z}{\ensuremath{\mathbb{Z}}}

\newcommand{\CC}{\mathcal{C}}

\newcommand{\CF}{\ensuremath{\mathcal{F}}}
\newcommand{\CG}{\ensuremath{\mathcal{G}}}

\newcommand{\CO}{\ensuremath{\mathcal{O}}}
\newcommand{\CZ}{\ensuremath{\mathcal{Z}}}
\newcommand{\ov}{\overline}

\newcommand{\ga}{\gamma}
\newcommand{\G}{\Gamma}
\newcommand{\T}{\theta}

\newcommand{\al}{\alpha}
\newcommand{\be}{\beta}
\newcommand{\cl}{\mbox{\normalfont{Cl}}}
\newcommand{\s}{\ensuremath{\mathbb{S}}}

\newcommand{\C}{\ensuremath{\mathcal{C}}}
\newcommand{\D}{\ensuremath{\mathcal{D}}}

\newcommand{\de}{\delta}
\newcommand{\De}{\Delta}

\def\p{\partial}
\def\e{\varepsilon}

\newtheorem {theorem} {Theorem} %[section]
\newtheorem {proposition} [theorem] {Proposition}
\newtheorem {corollary} [theorem] {Corollary}
\newtheorem {lemma} [theorem] {Lemma}

\newtheorem {example} {Example}

\newtheorem {claim} {Claim}
\newtheorem {mtheorem} {Theorem}

\begin{document}

\everymath{\displaystyle}
\allowdisplaybreaks
\title[]
{Persistence of periodic solutions for higher order perturbed differential systems via Lyapunov-Schmidt reduction}

\author{Murilo R. C\^{a}ndido$^{2}$}
\author{Jaume Llibre$^{2}$}
\author[D.D. Novaes]{Douglas D. Novaes$^{1}$}

\address{$^1$ Departamento de Matem\'{a}tica, Universidade
Estadual de Campinas, Rua S\'{e}rgio Buarque de Holanda, 651, Cidade Universit\'{a}ria Zeferino Vaz, 13083--859, Campinas, SP,
Brazil} \email{ddnovaes@ime.unicamp.br}
\address{$^{2}$ Departament de Matem\`{a}tiques, Universitat Aut\`{o}noma de
Barcelona, 08193 Bellaterra, Barcelona, Catalonia, Spain}
\email{candidomr@mat.uab.cat, jllibre@mat.uab.cat}


\maketitle

\begin{abstract}
In this work  we first provide sufficient conditions to assure the
persistence of some zeros of functions having the form
\begin{equation*}
g(z,\e)=g_0(z)+\sum_{i=1}^k \e^i g_i(z)+\CO(\e^{k+1}),
\end{equation*}
for $|\e|\neq0$ sufficiently small. Here $g_i:\D\rightarrow\R^n$,
for $i=0,1,\ldots,k$, are smooth functions being $\D\subset \R^n$ an
open bounded set. Then we use this result to  compute the
bifurcation functions which allow to study the periodic solutions of
the following $T$--periodic smooth differential system
\[
x'=F_0(t,x)+\sum_{i=1}^k \e^i F_i(t,x)+\CO(\e^{k+1}),
\quad (t,z)\in\s^1\times\D.
\]
It is assumed that the unperturbed differential system
has a sub-manifold of periodic solutions $\CZ$, $\dim(\CZ)\leq n$. We also study the case when the
 bifurcation functions have a continuum of zeros. Finally we provide the explicit
expressions of the bifurcation functions up to order 5.
\end{abstract}

\section{Introduction}

This work contains two main results. The first one (see Theorem \ref{LSt1}) provides sufficient conditions to assure the
persistence of some zeros of smooth functions $g:\R^n\times\R\rightarrow\R^n$  having the form
\begin{equation}\label{eq0}
g(z,\e)=g_0(z)+\sum_{i=1}^k \e^i g_i(z)+\CO(\e^{k+1}).
\end{equation}
The second one  (see Theorem \ref{PSt1}) provides sufficient conditions to assure the existence of periodic solutions
of the following differential system
\begin{equation}\label{ds0}
x'=F(t,z,\e)=F_0(t,x)+\sum_{i=1}^k \e^i F_i(t,x)+\CO(\e^{k+1}),
\quad (t,z)\in\s^1\times\D.
\end{equation}
Here $\s^1=\R/T$, for some $T>0$, and the assumption $t\in\s^1$ means that the system is $T$-periodic in the variable $t$. As usual $\delta_1(\e)=\CO\left(\delta_2(\e)\right)$ means that there exists a
constant $c_0>0$, which does not depends on $\e$, such that $|\delta_1(\e)|\leq c_0\, |\delta_2(\e)|$ for  $\e$ sufficiently small  (see \cite{SV}).

\smallskip

 It is assumed that either $g(z,0)$ vanishes in a submanifold of $\CZ\subset\mathcal{D},$ or that the unperturbed differential system $x'=F_0(t,x)$ has a submanifold $\CZ\subset\mathcal{D}$ of $T$-periodic solutions. In both cases $\dim(\CZ)\leq n.$ The second problem can be often reduced to the first problem, standing as its main motivation.

\smallskip

Regarding the first problem, assume that for some $z^*\in\CZ$,
$g(z^*,0)=0$.  We shall study the persistence of this zero for the function \eqref{eq0},
$g(x,\e),$ assuming that $|\e|\neq 0$ is sufficiently small. By persistence we mean the existence of continuous branches $\chi(\e)$ of simple zeros of $g(x,\e)$ (that is $g(\chi(\e),\e)=0$) such that $\chi(0)= z^*$. It is well known that
if the $n\times n$ matrix $\p_x g(z^*,0)$ (the Jacobian matrix of
the function $g$ with respect to the variable $x$ evaluated at
$x=z^*$) is nonsingular then, as a direct consequence of the
Implicit Function Theorem, there exists a unique smooth branch
$\chi(\e)$ of zeros of $g(x,\e)$ such
that $\chi(0)=x^*$. However if the matrix $\p_x g(x^*,0)$ is
singular (has non trivial kernel) we have to use the
Lyapunov--Schmidt reduction method to find branches
of zeros of $g$ (see, for instance, \cite{Ch}). Here we generalize some results from \cite{BFL,BGL1,LN}, providing a collection of functions $f_i$, $i=1,\ldots,k$, each one called {\it bifurcation function of order $i$}, which control the persistence of zeros contained in $\CZ$.

\smallskip

The second problem goes back to the works of Malkin \cite{Ma} and Roseau \cite{Ro}, whose have studied the persistence of periodic solutions for the differential system \eqref{ds0} with $k=1$. Let $x(t,z,\e)$ denote its solution such that $x(0,z,\e)=z$. In order to find initial conditions $z\in\mathcal{D}$ such that the solution $x(t,z,\e)$ is $T$-periodic we may consider the function $g(z,\e)=z-x(T,z,\e)$, and then try to use the results previously obtained from the first problem. Indeed, if $\CZ\subset\mathcal{D}$ is a submanifold of $T$-periodic solutions of the unperturbed system $x'=F_0(t,z)$ then $g(z,0)$ vanishes on $\CZ$. When $\dim(\CZ)=n$ this problem is studied at an arbitrary order of $\e$, see \cite{GGL,ddn}, even for nonsmooth systems. When $\dim(\CZ)<n$, this approach has already been used  in  \cite{BFL}, up to order 1, and in \cite{BGL1,BGL2}, up to order 2. In \cite{LN} this approach was used up to order 3 relaxing some hypotheses assumed in those previous 3 works. In \cite{GLWZ}, assuming the same hypotheses of \cite{BFL,BGL1,BGL2}, the authors studied this problem at an arbitrary order of $\e$. Here, following the ideas from \cite{ddn,LN}, we improve the results of \cite{GLWZ} relaxing some hypotheses and developing the method in a more general way.

\smallskip

In summary, in this paper we use the Lyapunov--Schmidt reduction
method for studying the zeros of functions like
\eqref{eq0} when the Implicit Function
Theorem cannot be directly applied. Another useful tool that we shall use to deal with this problem is the Browder
degree theory (see the Appendix B), which will allow us to provide estimates for these zeros. Then we apply these previous results for studying the periodic solutions of
differential systems like \eqref{ds0} through their bifurcation functions, provided by the higher order averaging theory.

\smallskip

This paper is organized as follows. In section \ref{mr} we state our main results: Theorem \ref{LSt1}, in subsection \ref{LSredu}, dealing with bifurcation of simple zeros of the equation $g(z,\e)=0$; and Theorem \ref{PSt1}, in subsection \ref{s12}, dealing with bifurcation of limit cycles of the differential equation $x'=F(t,z,\e)$. In sections \ref{proof1} and \ref{proof2} we prove Theorems \ref{LSt1} and \ref{PSt1}, respectively.  In section \ref{ApTB}, as an application of Theorem \ref{PSt1}, we study the birth of limit cycles in a 3D polynomial system. Finally, in section \ref{bf},  we study the case when the averaged functions have a continuum of zeros. In this last situation we also provide some results about the stability of the limit cycles.

\section{Statements of the main results}\label{mr}

Before we state our main results we need some preliminary concepts and
definitions. Given  $p$, $q$ and $L$   positive integers, $\ga_j=
(\ga_{j1},\ldots, \ga_{jp})\in \R^p$ for $j=1,\ldots,L$ and
$\ov{z}\in \R^p$. Let $G:\R^p\rightarrow\R^q$ be a sufficiently
smooth function, then the $L$-th Frechet derivative of  $G$ at $\ov
z$ is denoted by $\p^L G(\ov{z})$,  a symmetric $L$--multilinear
map, which applied to a ``product'' of $L$ $p$-dimensional vectors
denoted as $\bigodot_{j=1}^L\ga_j\in \R^{pL}$ gives
\[
\p^L G(\ov z)\bigodot_{j=1}^L\ga_j= \sum_{i_1,\ldots,i_L=1}^p \dfrac{\p^L
G(\ov z)}{\p b_{i_1}\cdots \p b_{i_L}}\ga_{1i_1}\cdots \ga_{Li_L}.
\]
The above expression is  indeed the G\^ateaux derivative
\begin{equation*}
\begin{aligned}
\p^L G(\ov z)\bigodot_{j=1}^L\ga_j&= \dfrac{\p}{\p\tau_1\p\tau_2\dots\p\tau_L}
G\left( \ov{z}+\tau_1\ga_1+\tau_2\ga_2+\dots+\tau_L\ga_L \right)
\Big|_{\tau_1=\dots=\tau_L=0} \\
&=\p\Big(\dots\p\big(\p G(\ov{z})\ga_1\big)\ga_2 \dots\Big)\ga_L.
\end{aligned}
\end{equation*}
We take $\p^0$ as the identity operator.

\subsection{The Lyapunov--Schmidt reduction method}\label{LSredu}

We consider the function
\begin{equation}\label{gz}
g(z,\e)=\sum_{i=0}^k \e^i g_i(z)+\CO(\e^{k+1}),
\end{equation}
where $g_i:\D\rightarrow\R^n$ is a $\C^{k+1}$ function, $k\geq 1$, for
$i=0,1,\ldots,k$, being $\D$ an open bounded subset of $\R^n$. For
$m < n$, let $V$ be an open bounded subset of $\R^m$ and
$\be:\cl(V)\rightarrow\R^{n-m}$ a $\C^{k+1}$ function, such that
\begin{equation}\label{Z}
\CZ=\{z_{\al}=(\al,\be(\al)):\,\al\in\cl(V)\}\subset\D.
\end{equation}
As usual  $\cl(V)$ denotes the closure of the set $V$.

As the main hypothesis we assume that
\begin{itemize}
\item[(H$_a$)] the function $g_0$ vanishes on the $m$--dimensional
submanifold $\CZ$ of $\D$.
\end{itemize}

\smallskip

Using the Lyapunov--Schmidt reduction method we shall develop the
bifurcation functions of order $i$, for $i=1,2,\ldots,k,$ which
control, for $|\e|\neq0$ small enough, the existence of branches of
zeros $z(\e)$ of \eqref{gz} bifurcating from $\CZ$, that is from
$z(0)\in\CZ$. With this purpose we introduce some notation. The
functions $\pi:\R^m\times\R^{n-m} \rightarrow\R^m$ and
$\pi^{\perp}:\R^m\times \R^{n-m} \rightarrow\R^{n-m}$  denote the
projections onto the first $m$ coordinates and onto the last $n-m$
coordinates, respectively. For a point $z\in \D$ we also consider
$z=(a,b)\in \R^m\times\R^{n-m}$.

\smallskip

For $i=1,2,\ldots,k$, we define the \textit{bifurcation functions}
$f_i:\cl(V)\rightarrow\R^m$ of order $i$
as\begin{equation}\label{fi} f_i(\al)=\pi
g_{i}(z_{\al})+\sum_{l=1}^i\sum_{S_l}\dfrac{1}{c_1!\,c_2!2!^{c_2}\cdots
c_l!l!^{c_l}}\p_b^L\pi
g_{i-l}(z_{\al})\bigodot_{j=1}^l\ga_j(\al)^{c_j},\quad\text{and}
\end{equation}
\begin{equation}\label{cF}
\CF^k(\al,\e)=\sum_{i=1}^k \e^i f_i(\al),
\end{equation}
where $\ga_i:V\rightarrow \R^{n-m}$, for $i=1,2,\ldots,k$, are
defined recurrently as
\begin{equation}\label{y}
\begin{array}{RL}
\ga_1(\al)=\!\!\!\!&-\Delta_{\al}^{-1}\pi^{\perp}g_1(z_\al)\quad\text{and}\vspace{0.3cm}\\
\ga_i(\al)=\!\!\!\!&-i!\Delta_{\al}^{-1}\Bigg(
\sum_{S'_i}\dfrac{1}{c_1!\,c_2!2!^{c_2}\cdots
c_{i-1}!(i-1)!^{c_{i-1}}}\p_b^{I'}\pi^{\perp} g_{0}(z_{\al})\bigodot_{j=1}^{i-1}
\ga_j(\al)^{c_j}\vspace{0.2cm}\\
&+ \sum_{l=1}^{i-1}\sum_{S_l}\dfrac{1}{c_1!\,c_2!2!^{c_2}\cdots
c_l!l!^{c_l}}\p_b^L\pi^{\perp}
g_{i-l}(z_{\al})\bigodot_{j=1}^l\ga_j(\al)^{c_j}\Bigg).
\end{array}
\end{equation}
Here $S_l$ is the set of all $l$-tuples of non--negative integers
$(c_1,c_2,\cdots,c_l)$ satisfying $c_1+2c_2+\cdots+lc_l=l$,
$L=c_1+c_2+\cdots+c_l$, and  $S'_i$ is the set of all
$(i-1)$-tuples of non--negative integers satisfying
$c_1+2c_2+\cdots+(i-1)c_{i-1}=i$, $I'=c_1+c_2+\cdots+c_{i-1}$ and
$\Delta_\al=\dfrac{\p\pi^{\perp} g_0}{\p b}(z_{\al})$.

\smallskip

We clarify that $S_0=S'_0=\emptyset,$ and when $c_j=0,$ for some
$j,$ then the term $\ga_j$ does not appear in the  ``product"
$\bigodot_{j=1}^l\ga_j(\al)^{c_j}$.

\smallskip

Recently in \cite{N} the Bell polynomials were used to provide an alternative formula for recurrences of kind \eqref{fi} and \eqref{y}. This new formula can make easier the computational implementation of the bifurcation functions \eqref{cF}.

\smallskip

The next theorem is the first main result of this paper. For sake of
simplicity, we take $f_0=0$.

\begin{mtheorem}\label{LSt1}
Let $\Delta_{\al}$ denote the lower right corner $(n-m)\times (n-m)$
matrix of the Jacobian matrix $D\,g_0(z_{\al})$. In additional to
hypothesis (H$_a$) we assume that
\begin{itemize}
\item[$(i)$] for each $\al \in\cl(V) $, $\det(\Delta_{\al})\neq0$;

\item[$(ii)$] for some $r\in\{1,\ldots,k\}$,  $f_1 = f_2 =\dots = f_{r-1}
= 0$ and $f_r$ is not identically zero;

\item[$(iii)$] there exists a small parameter $\e_0>0$  such that for each
$\e\in[-\e_0,\e_0]$ there exists $a_{\e}\in V$ satisfying
$\CF^k(a_{\e},\e )=0$;

\item[$(iv)$] there exist a constant $P_0>0$ and a positive integer  $l\leq
(k+r+1)/2$ such that
\begin{equation*}
\left|\p_\al\CF^k(a_\e,\e)\cdot \al\right|\geq P_0|\e|^l |\al|, \quad
\text{for} \quad \al\in V.
\end{equation*}
\end{itemize}
Then, for $|\e|\neq0$ sufficiently small, there exists $z(\e)$ such
that $g(z(\e),\e)=0$ with
$|\pi^{\perp}z(\e)-\pi^{\perp}z_{a_{\e}}|=\CO(\e)$ and
$|\pi\,z(\e)-\pi\,z_{a_{\e}}|=\CO(\e^{k+1-l}).$
\end{mtheorem}

Theorem \ref{LSt1} is proved in section \ref{proof1}.

\smallskip

In the next corollary we present a classical result in the literature, which is a direct consequence of Theorem \ref{LSt1}.

\begin{corollary}\label{ct1}
In addiction to hypothesis (H$_a$),  assume that $f_1 = f_2 =\dots =
f_{k-1} = 0,$ that is $r=k,$ and that for each $\al\in \cl(V)$,
$\det(\Delta_{\al})\neq0$. If there exists $\al^* \in V$ such that
$f_k(\al^*)=0$ and  $\det\left(Df_k (\al^*) \right)\neq
0$, then there exists a branch of zeros $z(\e)$ with $g(z(\e),\e)=0$
and $|z(\e)-z_{\al^*}|=\CO(\e)$.
\end{corollary}

Corollary \ref{ct1} is proved in section \ref{proof1}

\subsection{Continuation of periodic solutions}\label{s12}

As an application of Theorem \ref{LSt1} we study  higher order
bifurcation of periodic solutions of the following $T$--periodic
$\C^{k+1}$, $k\geq 1$, differential system
\begin{equation}\label{s1}
x'=F(t,x,\e)=F_0(t,x)+\sum_{i=1}^k \e^i F_i(t,x)+\CO(\e^{k+1}),
\quad (t,z)\in\s^1\times\D.
\end{equation}
Here $\s^1=\R \big / (T\,\Z)$ with $T\Z=\lbrace T, 2T, \dots
\rbrace$ and the prime denotes derivative with respect to time $t$.
Now the manifold $\CZ$, defined in \eqref{Z}, is seen as a set of
initial conditions of the unperturbed system

\begin{equation}\label{ups}
x'(t)=F_0(t,x).
\end{equation}

In fact we shall assume that all solutions of the unperturbed system
starting at points of $\CZ$ are $T$-periodic, recall that the dimension of
$\CZ$ is $m\leq n$. Formally,  let
$x(\cdot,z,0):[0,t_z)\rightarrow\mathbb{R}^n$ denote the solution of
\eqref{ups} such that $x(0,z,0)=z$, we assume that
\begin{itemize}
\item[(H$_b$)] $\CZ \subset \D$ and for each $\al \in\cl(V)$ the
solution $x(t,z_\al,0)$ of \eqref{ups} is $T$-periodic.
\end{itemize}

As usual $x(\cdot,z,\e):[0,t_{(z,\e)})\rightarrow\mathbb{R}^n$
denotes a solution of system \eqref{s1} such that $x(0,z,\e)=z$.
Moreover, let $Y(t,z)$ be a fundamental matrix solution of the
linear differential system
\begin{equation}\label{vs}
y'=\dfrac{\partial}{\partial x}F_0(t,x(t,z,0))y.
\end{equation}
For sake of simplicity when $z=z_\al \in \CZ$ we denote $Y_\al(t)=Y(t,z_\al)$.

\smallskip

Given
fundamental matrix solution $Y(t,z)$, the \textit{averaged functions
of order $i$}, $g_i:\cl(V)\rightarrow \R^n$,  $i=1,2,\ldots,k,$ of system \eqref{s1} is defined as
\begin{equation}\label{smoothgi}
g_{i}(z)=Y(T,z)^{-1}\dfrac{y_i(T,z)}{i!},
\end{equation}
where
\begin{equation}\label{smoothyi}
\begin{array}{RL}
y_1(t,z)=\!\!\!\!&Y(t,z)\!\!\int_{0}^t\!\!Y(s,z)^{-1}F_1(s,x(s,z,0))ds,\vspace{0.2cm}\\
y_i(t,z)=\!\!\!\!&i!\,Y(t,z)\!\!\int_{0}^t\!\!Y(s,z)^{-1}\Bigg(\!\!F_i(s,x(s,z,0))\\
\!\!\!\!&+\sum_{S'_i}\dfrac{1}{b_1!\,b_2!2!^{b_2}\cdots
b_{i-1}!(i-1)!^{b_{i-1}}}\p^{I'}F_{0}(s,x(s,z,0))\bigodot_{j=1}^{i-1}y_j(s,z)^{b_j}\\
\!\!\!\!&+\sum_{l=1}^{i-1}\sum_{S_l} \dfrac{1}{b_1!\,b_2!2!^{b_2}\cdots
b_l!l!^{b_l}}\p^LF_{i-l}(s,x(s,z,0))
\bigodot_{j=1}^ly_j(s,z)^{b_j}\!\!\Bigg)ds.
\end{array}
\end{equation}
Using now the functions $g_i$ as stated in \eqref{smoothgi} we define the functions $f_i$, $\CF^k$, and
$\gamma_i$ given by  \eqref{fi}, \eqref{cF}, and \eqref{y}, respectively.

\smallskip

Recently in \cite{N} the Bell polynomials were used to provide an alternative formula for the recurrence \eqref{smoothyi}. This new formula can also make easier the computational implementation of the bifurcation functions \eqref{smoothgi}.

\smallskip

The next theorem is the second main result of this paper.  Again,
for sake of simplicity, we take $f_0=0$.

\begin{mtheorem}\label{PSt1}
Let $\Delta_{\al}$ denote the lower right corner $(n-m)\times (n-m)$
matrix of the matrix $Y(0,z)^{-1}-Y(T,z)^{-1}$. In
additional to hypothesis (H$_b$) we assume that
\begin{itemize}
\item[$(i)$] for each $\al \in\cl(V) $, $\det(\Delta_{\al})\neq0$;

\item[$(ii)$] for some $r\in\{1,\ldots,k\}$,  $f_1 = f_2 =\dots = f_{r-1}
= 0$ and $f_r$ is not identically zero;

\item[$(iii)$] there exists a small parameter $\e_0>0$  such that for each
$\e\in[-\e_0,\e_0]$ there exists $a_{\e}\in V$ satisfying $\CF^k(a_{\e},\e )=0$;

\item[$(iv)$] there exist a constant $P_0>0$ and a positive integer $l\leq
(k+r+1)/2$ such that
\begin{equation*}
\left|\p_\al\CF^k(a_\e,\e)\cdot \al\right|\geq P_0|\e^l||\al|, \quad
\text{for} \quad \al\in V.
\end{equation*}
\end{itemize}
Then, for $|\e|\neq0$ sufficiently small, there exists a
$T$-periodic solution $\varphi(t,\e)$ of system \eqref{s1} such that
$|\pi\,\varphi(0,\e)-\pi\,z_{a_{\e}}|=\CO(\e^{k+1-l}),$ and
$|\pi^{\perp}\varphi(0,\e)-\pi^{\perp}z_{a_{\e}}|=\CO(\e)$.
\end{mtheorem}

Theorem \ref{PSt1} is proved in section \ref{proof2}.

\smallskip

In the next corollary we present a classical result in the literature, which is a direct consequence of Theorem \ref{PSt1}.

\begin{corollary}\label{ct1a}
In addiction to hypothesis (H$_b$) we assume that $f_1 = f_2 =\dots =
f_{k-1} = 0$, $r=k$ and that for each $\al\in \cl(V)$,
$\det(\Delta_{\al})\neq0$. If there exists $\al^* \in V$ such that
$f_k(\al^*)=0$ and  $\det\left(Df_k (\al^*) \right)\neq
0$, then there exists a $T$-periodic solution $\varphi(t,\e)$ of
\eqref{s1} such that  $|\varphi(0,\e)-z_{\al^*}|=\CO(\e)$.
\end{corollary}

Corollary \ref{ct1a} is proved in section \ref{proof2}. An
application of Theorem B is performed in Section \ref{ApTB}.

\smallskip

It is worth to emphasize that Theorem \ref{PSt1} is still true when $m=n$. In fact,  assuming that
$V$ is an open subset of $\R^n$ then $\CZ =\cl(V)\subset \D$ and the
projections $\pi$ and $\pi^\perp$ become the identity and the null
operator respectively. Moreover, in this case the bifurcation functions
$f_i:V\rightarrow \R^n$, for $i=1,2,\dots,k$, are the averaged functions
$f_i(\al)=g_i(\al)$ defined in \eqref{smoothgi}. Thus we have the following corollary, which recover the main result from \cite{ddn}.

\begin{corollary}
Consider $m=n$,  $z_\al=\al\in\CZ$ and the hypothesis (H$_b$).  Thus
the result of Theorem  \ref{PSt1} holds without any assumption about
$\Delta_\al$.
\end{corollary}

\section{Proof of Theorem \ref{LSt1} and Corollary \ref{ct1}}\label{proof1}

A useful tool to study the zeros of a function is the Browder
degree (see the Appendix B for some of their properties).  Let $g \in
C^1(\D)$, $\cl(V)\subset \D$ and $\mathbf{Z}_g=\lbrace z\in V :
g(z)=0\rbrace$. We also assume that $J_g(z)\neq 0$ for all $z\in
\mathbf{Z}_g$, where $J_g(z)$ is the Jacobian determinant of  $g$ at
$z$. This assures that the set $\mathbf{Z}_g$ is formed by a finite number of isolated
points. Then the Brouwer degree of $g$ at $0$ is
\begin{equation}\label{defbr}
d_B(g,V,0)=\sum_{z\in\mathbf{Z}_g}\mathrm{sign}\left(J_g(z)\right).
\end{equation}

As one of the main properties of the Brouwer degree we have that: ``if $d(f,V,0)\neq 0$ then there exists
$x_0\in V$ such that $f(x_0)=0$''(see item (i) of Theorem \ref{ApAt1} from Appendix B).

\smallskip

The next result
is a key lemma for proving Theorem \ref{LSt1}.

\begin{lemma}\label{NNL}
Let $V$ be an open bounded subset of $\R^m$. Consider the continuous
functions $f_i:\cl(V)\to \mathbb{R}^n$, $i=0,1,\cdots,\kappa$, and
$f,g,r:\cl(V)\times [-\e_0,\e_0] \rightarrow \mathbb{R}^n$ given by
\[
g(x,\e)=f_0(x)+\e f_1(x)+\cdots+\e^\kappa f_\kappa(x) \mbox{ and }
f(x,\e)=g(x,\e)+\e^{\kappa+1}r(x,\e).
\]
Let $V_\e \subset V$, $R=\max\{\vert r(x,\e) \vert :
(x,\e)\in\cl(V)\times [-\e_0,\e_0]\}$ and assume that $\vert g(x,\e)
\vert> R\vert \e\vert^{\kappa+1}$ for all $ x \in \p V_\e $ and
$\e\in[-\e_0,\e_0]\setminus\{0\}$. Then for each
$\e\in[-\e_0,\e_0]\setminus\{0\}$ we have $d_B\left(f(\cdot,
\e),V_\e,0 \right)=d_B\left(g(\cdot, \e),V_\e,0  \right)$.
\end{lemma}

\begin{proof}
For a fixed $\e\in[-\e_0,\e_0]\setminus\{0\}$, consider a continuous
homotopy between $g(\cdot,\e)$ and $f(\cdot,\e)$ given by
$g_t(x,\e)=g(x,\e)+t\left(
f(x,\e)-g(x,\e)\right)=g(x,\e)+t\,\e^{\kappa+1}r(x,\e)$.  We claim that
$0\not\in g_t(\p V_\e, \e)$ for every $t\in[0,1]$. As usual $\p
V_\e$ denotes the boundary of the set $V_\e$. Indeed,  assuming that
$0\in g_{t_\e}(\p V_\e, \e),$ for some $t_\e\in[0,1],$ we may find
$x_\e\in\p V_\e$ such that $g_{t_\e}(x_\e,\e)=0$ and, consequently,
$g(x_\e,\e)=-t_\e\e^{\kappa+1} r(x_\e,\e)$. Thus $\vert g(x_\e,\e) \vert
\leq R\vert\e \vert^{\kappa+1}$, which contradicts the hypothesis $\vert
g(x_\e,\e) \vert > R\vert\e\vert^{\kappa+1}$. From Theorem \eqref{ApAt1}
(iii) we conclude that $d_B(g_t(\cdot,\e),V_\e,0)$ is constant for
$t\in[0,1]$ and then  $d_B\left(f(\cdot, \e),V_\e,0
\right)=d_B\left(g(\cdot, \e),V_\e,0  \right)$.
\end{proof}

The above lemma  provides a stratagem to track  zeros of the perturbed function $f(x,\e)$ using a shrinking neighborhood around the zeros of $g(x,\e)$ that preserves its Brouwer degree. The way how it works can be blurry at the first moment, so to make it clear we present the following example:

\begin{example}Consider the real function $f(x,\e)=g(x,\e)+\e^2r(x,\e)$ with $(x,\e) \in [-1,\, 1]\times[-\e_0,\e_0]$, $g(x,\e)=x^2-\e x$, and  $|r(x,\e)|\leq1/5$. The function $g(x,\e)$ has two zeros $a=0$ and $a_\e=\e$. Taking $V_\e=(\e/2,\,3\e/2 )$ we have that, for $|\e|\neq0$ sufficiently small,  $a_\e \in V_\e$ and $d_B\left(g(\cdot, \e),V_\e,0
\right)=1$ $($see Definition \eqref{defbr}$)$. Furthermore $\p V_\e=\left\lbrace \e/2,\,3 \e/2 \right\rbrace$, $|g(\e/2, \e)|=\e^2/4$, and $|g(3\e/2, \e)|=3 \e^2/4$. Thus $|g(x, \e)|>\e^2/5\geq\e^2\max\{\vert r(x,\e) \vert :
(x,\e)\in [0,1]\times [-\e_0,\e_0]\}$. Therefore from the previous lemma we know that $d_B\left(f(\cdot, \e),V_\e,0\right)=1$. From the above property of the Brouwer degree we conclude that there exists $\al_{\e}\in V_{\e}$ such that $f(\al_{\e},\e)=0$.
\end{example}

Now we recall the Fa\'{a} di Bruno's Formula (see \cite{J}) about the
$l^{th}$ derivative of a composite function.

\smallskip

\noindent{\bf Fa\'{a} di Bruno's Formula} {\it If $u$ and $v$ are
functions with a sufficient number of derivatives, then
\[
\dfrac{d^l}{dt^l}u(v(t))=\sum_{S_l}\dfrac{l!}{b_1!\,b_2!2!^{b_2}\cdots
b_l!l!^{b_l}}u^{(L)}(v(t))\bigodot_{j=1}^l v^{(j)}(t)^{b_j},
\]
where $S_l$ is the set of all $l$--tuples of non--negative integers
$(b_1,b_2,\cdots,b_l)$ which are solutions of the equation
$b_1+2b_2+\cdots+lb_l=l$ and $L=b_1+b_2+\cdots+b_l$.}

\smallskip

The remainder of this section consists in the proof of Theorem \ref{LSt1}, which is split in several claims, and the proof Corollary \ref{ct1}

\smallskip

\begin{proof}[Proof of Theorem \ref{LSt1}]
We consider $g=(\pi g,\pi^{\perp} g)$,  $g_i=(\pi
g_i,\pi^{\perp} g_i)$ for $i=0,1,2,\dots, k$, and
$z=(a,b)\in\R^m\times\R^{n-m}$ for $z\in \D$. So
\[
\dfrac{\p g}{\p z}(z_{\al},0)=D\,g_0(z_{\al})=
\left(\begin{array}{CC}
\dfrac{\p\pi g_0}{\p a}(z_{\al})&\dfrac{\p\pi g_0}{\p
b}(z_{\al})\\[0.4cm]
\dfrac{\p\pi^{\perp}g_0}{\p a}(z_{\al})&\dfrac{\p\pi^{\perp}g_0}{\p
b}(z_{\al}) \end{array}\right).
\]
We write $\Delta_{\al}=\dfrac{\p\pi^{\perp}g_0}{\p b}(z_{\al})$.
From hypotheses,
$\pi^{\perp}g(\al,\be(\al),0)=\pi^{\perp}g_0(z_{\al})=0$ and
\[
\det\left(\dfrac{\p\pi^{\perp}g}{\p b}(\al,\be(\al),0)\right)=
\det\left(\dfrac{\p\pi^{\perp}g_0}{\p b}(z_{\al})\right)=
\det\left(\Delta_{\al}\right)\neq0.
\]
Thus applying the {\it Implicit Function Theorem}  it follows that
there exists an open neighborhood $U\times(-\e_1,\e_1)$ of
$\cl(V)\times\{0\}$ with $\e_1\leq\e_0$, and a $\C^{k+1}$ function
$\ov{\be}:U\times(-\e_1,\e_1)\rightarrow\R^{n-m}$ such that
$\pi^{\perp}g(a,\ov{\be}(a,\e),\e)=0$ for each $(a,\e)\in
U\times(-\e_1,\e_1)$ and $\ov{\be}(\al,0)=\be(\al)$ for every
$\al\in\cl(V)$.

\smallskip

From here, this proof will be split in several claims.

\begin{claim}\label{c1}
The equality $(\p^i\ov{\be}/\p\e^i)(\al,0)=\gamma_i(\al)$ holds for
$i=1,2,\ldots,k.$
\end{claim}

Firstly, it is easy to check that
$(\p\ov{\be}/\p\e)(\al,0)=\gamma_1(\al)$. Now, for some fixed $i\in\{1,2,\ldots,k\},$ we assume by induction hypothesis that  $(\p^s\ov{\be}/\p\e^s)(\al,0)=\gamma_s(\al)$  for $s=1,\dots,i-1$. In what
follows we prove the claim for $s=i$. Consider
$$
\pi^{\perp}g(\al,\ov{\be}(\al,\e),\e)=\sum_{i=0}^k \e^i
\pi^{\perp}g_i(\al,\ov{\be}(\al,\e))+\CO(\e^{k+1})=0.
$$
Expanding each function $\e \mapsto \pi^\perp g_i\big(\al,
\ov{\beta}(\al,\e) \big)$ in Taylor series we obtain

\begin{align}\label{c10}
\pi^{\perp}g(\al,\ov{\be}(\al,\e),\e)=&\sum_{i=0}^k\left(\e^i
\sum_{l=0}^i \dfrac{1}{l!}\dfrac{\partial^l}{\partial
\e^l}\pi^{\perp}g_{i-l}\left(\al,\ov{\be}(\al,\e)\right)\Big|_{\e=0}\right)
\\
&+\CO(\e^{k+1})=0 \nonumber.
\end{align}
Applying the the Fa\`{a} di Bruno's formula we obtain
\begin{align}\label{c12}
\dfrac{\partial^l}{\partial
\e^l}\pi^{\perp}g_{i-l}\left(\al,\ov{\be}(\al,\e)\right)\Big|_{\e=0}=
&\sum_{S_l}\left(\dfrac{l!}{b_1!\,b_2!2!^{b_2}\cdots
b_l!l!^{b_l}}\partial_b^L \pi^{\perp}g_{i-l}\left(\al,\ov{\be}(\al,0)\right)\right.\\
&\left.\bigodot_{j=1}^l \dfrac{\partial^j}{\partial
\e^j}\ov{\be}(\al,0)^{b_j}\right). \nonumber
\end{align}
Substituting \eqref{c12} in \eqref{c10} we get
\begin{align}
\pi^{\perp}g(\al,\ov{\be}(\al,\e),\e)=&\sum_{i=0}^k \e^i\left(
\sum_{l=0}^i \sum_{S_l}\dfrac{1}{b_1!\,b_2!2!^{b_2}\cdots
b_{l}!l!^{b_{l}}}\partial_b^{L}
\pi^{\perp}g_{i-l}\left(\al,\ov{\be}(\al,0)\right)\right.
\nonumber\\ &\left.\bigodot_{j=1}^{l} \dfrac{\partial^j}{\partial
\e^j}\ov{\be}(\al,0)^{b_j} \right) +\CO(\e^{k+1})=0. \nonumber
\end{align}
Since the previous equation is equal to zero for $|\e|$ sufficiently
small, the coefficients of each power of $\e$  vanish. Then  for $0
\leq i \leq k$ and $(\al,\e)\in U\times(-\e_1,\e_1)$ we have
\begin{align*}
&\sum_{l=0}^i \sum_{S_l}\dfrac{1}{b_1!\,b_2!2!^{b_2}\cdots
b_{l}!l!^{b_{l}}}\partial_b^{L}
\pi^{\perp}g_{i-l}\left(\al,\ov{\be}(\al,0)\right)\bigodot_{j=1}^{l}
\dfrac{\partial^j}{\partial \e^j}\ov{\be}(\al,0)^{b_j} =0.
\end{align*}
This equation can be rewritten as
\begin{align}\label{c11}
0=&\sum_{l=0}^{i-1} \sum_{S_l}\dfrac{1}{b_1!\,b_2!2!^{b_2}\cdots
b_{l}!l!^{b_{l}}}\partial_b^{L}
\pi^{\perp}g_{i-l}\left(\al,\ov{\be}(\al,0)\right)\bigodot_{j=1}^{l}
\dfrac{\partial^j}{\partial \e^j}\ov{\be}(\al,0)^{b_j} \nonumber
\\
&+\sum_{S'_i}\dfrac{1}{b_1!\,b_2!2!^{b_2}\cdots
b_{i-1}!(i-1)!^{b_{i-1}}}\partial_b^{I'}
\pi^{\perp}g_{0}\left(\al,\ov{\be}(\al,0)\right)\bigodot_{j=1}^{i-1}
\dfrac{\partial^j}{\partial \e^j}\ov{\be}(\al,0)^{b_j}
\\
&+\dfrac{1}{i!}\partial_b
\pi^{\perp}g_{0}\left(\al,\ov{\be}(\al,0)\right)
\dfrac{\partial^i}{\partial \e^i}\ov{\be}(\al,0). \nonumber
\end{align}
Here $S'_i$ is the set of all  $(i-1)$-tuples of non--negative
integers satisfying $b_1+2b_2+\cdots+(i-1)b_{i-1}=i$,
$I'=b_1+b_2+\cdots+b_{i-1}$. Finally, using the induction hypothesis,
equation \eqref{c11} becomes
\begin{align*}
\dfrac{\p^i\be}{\p \e^i}(\al,0)&=-i!\Delta_\al^{-1}
\left(\sum_{S'_i}\dfrac{1}{b_1!\,b_2!2!^{b_2}\cdots
b_{(i-1)}!(i-1)!^{b_{i-1}}}\p^{I'}_b\pi^{\perp}
g_0(z_\al)\bigodot_{j=1}^{i-1}\gamma_j(\alpha)^{b_s}\right. \\
&\left.+\sum_{l=0}^{i-1}
\sum_{S_l}\dfrac{1}{b_1!\,b_2!2!^{b_2}\cdots
b_{l}!l!^{b_{l}}}\p^{L}_b\pi^{\perp}
g_{i-l}(z_\al)\bigodot_{j=1}^{l}
\gamma_j(\alpha)^{b_s}\right)=\gamma_i(\al).
\end{align*}
This concludes the proof of Claim \ref{c1}.

\begin{claim}\label{c2}
Let $\de:U\times(-\e_1,\e_1)\rightarrow\R^m$ be the $C^{k+1}$
function defined as $\de(\al,\e)= \pi g(\al,\ov{\be}(\al,\e),\e)$.
Then the equality $(\p^i\de/\p\e^i)(\al,0)=i!f_i(\al)$ holds for
$i=1,2,\ldots,k.$
\end{claim}
From \eqref{gz} the function $\delta$ reads
\[
\de(\al,\e)=\sum_{j=0}^k \e^j \pi g_j(\al,\ov{\be}(\al,\e))+\CO(\e^{k+1}).
\]
So computing its $i$th-derivative, $0\leq i \leq k$, in the
variable $\e$, we get
$$
\dfrac{\p^i\de}{\p \e^i}(\al,\e)=\sum_{j=0}^i
\sum^i_{q=0}\binom{i}{q}(\e^j)^{(i-q)} \dfrac{\p^q\pi g_j}{\p
\e^q}(\al,\ov{\be}(\al,\e))+\CO(\e).
$$
Taking $\e=0$ and $l=i-j$ we obtain
$$
\dfrac{\p^i\de}{\p \e^i}(\al,0)=\left.\sum_{l=1}^i\dfrac{i!}{l!}
\dfrac{\p^l\pi g_{i-l}}{\p
\e^l}(\al,\ov{\be}(\al,\e))\right|_{\e=0}+i!\pi g_i(z_\al).
$$
Finally using the Fa\`{a} di Brunno's formula and Claim \ref{c1} we have
\begin{align*}
\dfrac{\p^i\de}{\p \e^i}(\al,0)&=\sum_{l=1}^i\dfrac{i!}{l!}\sum_{S_l}
\dfrac{l!}{c_1!c_2!2!^{c_2}\dots c_l!l!^{c_l}}\p^L_b\pi
g_{i-l}(z_\al)\bigodot_{s=1}^{l}\gamma_s(\al)^{c_s}+i!\pi
g_i(z_\al)\\
&=i!f_i(\al).
\end{align*}
This concludes the proof of Claim 2.

\medskip

Using Claim \ref{c2} the function $\de(\al,\e)$ can be expanded in
power series of $\e$ as
\begin{equation*}
\de(\al,\e)=\sum_{i=0}^k\dfrac{\e^i}{i!}\dfrac{\p^i\de}{\p
\e^i}(\al,0)+\CO(\e^{k+1})=\CF^k(\al,\e)+\CO(\e^{k+1}),
\end{equation*}
and, from hypothesis $(ii)$, we have
\begin{equation}\label{tsdelta}
\widetilde{\de}(\al,\e):=\dfrac{\de(\al,\e)}{\e^{r}}=
\CG^k(\al,\e)+\CO(\e^{k-r+1}),
\end{equation}
where $\CG^k(\al,\e)=f_r(\al)+\e
f_{r+1}(\al)+\ldots+\e^{k-r}f_{k}(\al).$ Obviously the equations
$\de(\al,\e)=0$ and $\widetilde{\de}(\al,\e)=0$ are equivalent for
$\e\neq0$.

\smallskip

Denote $R(\e_0)=\max\{|\widetilde{\delta}(\al,\e)-
\CG^k(\al,\e)|:\,(\al,\e)\in\cl(V)\times [-\e_0,\e_0]\}$. From the
continuity of the functions $\widetilde{\delta}$ and $\CG^k$ and
from the compactness of the set $\cl(V)\times [-\e_0,\e_0]$ we know that
$R(\e_0)<\infty$ and $R(0)=0$. In order to study the zeros of
$\widetilde{\de}(\al,\e)$ we use Lemma \ref{NNL},  for proving the
following claim.

\begin{claim}\label{c3}
Consider $a_\e\in V$ as given in hypothesis $(iii)$ and
$\e\in[-\e_0,\e_0]$. Then, there exist $\e_0>0$ sufficiently small and, for each
$\e\in[-\e_0,\e_0]$, a neighborhood $V_\e \subset V$ of
$a_\e$  such that $|\CG^k(\overline{\al},\e)|> R(\e_0) |\e^{k-r+1}|$ for
all $\overline{\al} \in \partial V_\e.$ Moreover
$V_\e=B(a_\e,Q|\e|^{k+1-l})$ for some $Q>0$.
\end{claim}

The parameter $\e_0>0$ will be chosen later on. Given $\e\in[-\e_0,\e_0]$, since $\CG^k(\al,\e)$ is a $\CC^{k+1}$ function, $k\geq 1$, we have that
\begin{equation}\label{Taylor}
\CG^k(a_\e +h,\e)= \partial_\al \CG^k(a_\e ,\e)h+\rho(h),\quad \rho(h)=\CO(|h|^2),
\end{equation}
for every $h \in\R^m$ such that $[a_\e,\, a_\e+h]\subset V$. Moreover, hypotheses  $(ii)$ and $(iv)$
imply that
\begin{equation}\label{G}
\left|\p_\al\CG^k(a_\e,\e)\cdot \al\right|\geq P_0|\e|^{l-r} |\al|
\quad \text{for} \quad \al\in V.
\end{equation}
Combining expressions \eqref{Taylor} and \eqref{G} we obtain the following inequality:
\begin{align}\label{d1}
| \CG^k(a_\e+h ,\e)|\geq \left(P_0 -|\e|^{r-l}\dfrac{|\rho(h)|}{|h|}
\right)|\e|^{l-r}|h|.
\end{align}

Take
$V_\e=B(a_\e,Q|\e|^{k+1-l})\subset V$. A point
$\overline{\al}_{\e}\in \partial V_\e$ reads
$\overline{\al}_{\e}=a_\e+h_{\e}$, where $h_{\e}=u Q|\e|^{k+1-l} \in
\mathbb{R}^m$ and $|u|=1$.  Moreover, since $\rho(h)=\CO(|h|^2)$ we get
\[
|\e|^{r-l}\dfrac{|\rho(h_{\e})|}{|h_{\e}|}
=|\e|^{r-l}\CO\big(Q|\e|^{k+1-l}\big)=\CO\big(Q|\e|^{k+r+1-2l}\big).
\]
From hypothesis $(iii),$ $k+r+1-2l\geq0$. So, in particular,
$\CO\big(Q|\e|^{k+r+1-2l}\big)=\CO(Q)$. Thus, from definition of the symbol $\CO$, there exists $c_0>0$, which does not depend on $\e$ and $Q$,
such that $|\e|^{r-l}|\rho(h_{\e})|/|h_{\e}|\leq c_0 Q.$ So the inequality \eqref{d1} reads
\begin{equation*}
| \CG^k(a_\e+h_\e ,\e)|\geq \left(P_0 - Q c_0 \right)Q|\e|^{k-r+1}.
\end{equation*}
Note that the polynomial $\mathcal{P}(Q)=\left(P_0 - Q c_0 \right)Q$ is positive for $0<Q<P_0/c_0$ and reach its maximum at $Q^*=P_0/(2c_0)$. Moreover $\mathcal{P}(Q^*)=P_0^2/(4c_0)$. Since $R(0)=0$, there exists $\e_0>0$ small enough in order that $R(\e_0)< P_0^2/(4c_0)=\mathcal{P}(Q^*)$. Accordingly, taking $Q=Q^*$ it follows that $|\CG^k(\overline{\al},\e)|> R(\e_0) |\e^{k-r+1}|$ for all $\overline{\al}
\in \partial V_\e$ and $\e\in[-\e_0,\e_0]$. This concludes the proof of Claim.
\smallskip

Applying Lemma \ref{NNL} for $g=\widetilde{\delta}$, as defined in \eqref{tsdelta}, $\kappa=k-r$, and $V_\e=B(a_\e,Q|\e|^{k+1-l})$ we conclude that $d_B\big(\widetilde{\delta}(\cdot,
\e),V_\e,0 \big)=d_B\big(\CG^k(\cdot, \e),V_\e,0 \big) \neq 0$.
Finally, denoting $z(\e)=\big(\al(\e),\ov{\be}\big(\al(\e),$
$\e\big)\big)$ it follows that $g(z(\e ),\e )=0$.

\smallskip

Moreover, let $z_{a_\e }=\left(a_\e, \beta(a_\e) \right)$, then  $
\vert\pi z(\e)- \pi z_{a_\e } \vert= | \al(\e)-a_\e |=\CO\left(
\e^{k+1-l} \right)$ and, since $\ov{\beta}$ is Lipschtiz,
\[
\left\vert
\pi^\perp z(\e)- \pi^\perp z_{a_\e } \right\vert=\left\vert
\ov{\beta}(\al(\e),\e)-\ov{\beta}(a_\e,0) \right\vert\leq L|(\al(\e),\e)-(a_{\e},0)|=\CO(\e).
\]
This
concludes the proof of Theorem \ref{LSt1}.
\end{proof}

\begin{proof}[Proof of Corollary \ref{ct1}]
The basic idea of the proof is to show that $\CF^k(\al)$ satisfies
all the hypotheses of Theorem \ref{LSt1}. From hypotheses,
$\CF^k(\al,\e)=\e^k f_k(\al)$ and
$Df_k(\al^*)=\e^{-k}\partial_\al\CF^k(\al^*,\e)$ is a
homeomorphism on $\mathbb{R}^n$. Thus there exist constants $b,c>0$
such that
\[
b|\al|<\left\vert Jf_k(\al^*) . \al \right\vert=\left\vert
\dfrac{1}{\e^k}\partial_\al\CF^k(\al^*,\e).\al\right\vert<c|\al|,
\]
for all $\al \in \mathbb{R}^m$. Therefore $b\left\vert\e^k \right\vert|\al|<\left\vert
\partial_\al\CF^k(\al^*,\e).\al\right\vert<c\left\vert\e^k
\right\vert|\al|$, which implies that $\CF^k(\al^*)$ satisfies  hypothesis $(iii)$ of Theorem \ref{LSt1}, with $l=r=k$. Indeed $(k+r+1)/2=k+1/2>k=l$.  Hence the proof follows directly
from Theorem  \ref{LSt1}.
\end{proof}

\section{Proof of Theorem \ref{PSt1} and Collorary \ref{ct1a}}\label{proof2}

The next result is needed in the proof of Theorem \ref{PSt1}.

\begin{lemma}[Fundamental Lemma]\label{l1}
Let $x(t,z,\e)$ be the solution of the $T$-periodic $\CC^{k+1}$ differential system \eqref{s1} such that
$x(0,z,\e)=z$. Then the equality
\[
x(t,z,\e)=x(t,z,0)+\sum_{i=1}^{k}\e^{i}\dfrac{y_i(t,z)}{i!}+\CO(\e^{k+1})
\]
holds for $(t,z)\in\s^1\times\D$.
\end{lemma}

\begin{proof}
The solution $x(t,z,\e)$ can be written as
\begin{equation}\label{x}
\begin{array}{L}
x(t,z,\e)=z+\sum_{i=0}^k\e^i\int_{0}^tF_i(s,x(s,z,\e))ds+\CO(\e^{k+1}),
\quad \textrm{and}\vspace{0.2cm}\\ x(t,z,0)=z+\int_{0}^t
F_0(s,x(s,z,0))ds.
\end{array}
\end{equation}
Moreover the result about differentiable dependence on parameters
implies that $\e\mapsto x(t,z,\e)$ is a $\CC^{k+1}$ map. So, for
$i=0,1,\ldots,k-1$, we compute the Taylor expansion of
$F_i(t,x(t,z,\e))$ around $\e=0$ as
\begin{equation}\label{exp1}
F_i(t,x(t,z,\e))=F_i\left(t,x(t,z,0)\right)+\sum_{l=1}^{k-i}\dfrac{\e^l}{l!}
\left(\dfrac{\p^l}{\p
\e^l}F_i(t,x(t,z,\e))\right)\Bigg|_{\e=0}+\CO(\e^{k-i+1}).
\end{equation}
Using the Fa\'{a} di Bruno's formula to compute the
$l$--derivatives of $F_i(t,x(t,z,\e))$ in the variable $\e$ we get
\begin{equation}\label{exp2}
\begin{array}{RL}
\dfrac{\p^l}{\p
\e^l}F_i(t,x(t,z,\e))\Bigg|_{\e=0}\!\!\!\!=\sum_{S_l}\dfrac{l!}{b_1!\,b_2!
2!^{b_2}\cdots b_l!l!^{b_l}}\p^L
F_i(t,x(t,z,0)) \bigodot_{j=1}^ly_j(t,z)^{b_j},
\end{array}
\end{equation}
where
\begin{equation}\label{yj}
y_j(t,z)=\left(\dfrac{\p^j}{\p\e^j}x(t,z,\e)\right)\Bigg|_{\e=0}.
\end{equation}
Substituting \eqref{exp2} in \eqref{exp1} the Taylor expansion of $F_i(s,x(t,z,\e))$ becomes
\begin{equation}\label{exp3}
\begin{array}{RL}
F_i(s,x(s,z,\e))=&F_i\left(s,x(s,z,0)\right)\\
&+\sum_{l=1}^{k-i}\sum_{S_l}\dfrac{\e^l}{b_1!\,b_2!2!^{b_2}\cdots
b_l!l!^{b_l}}\p^L
F_i\left(s,x(s,z,0)\right)\\
&\bigodot_{j=1}^ly_j(s,z)^{b_j}+\CO(\e^{k-i+1}),
\end{array}
\end{equation}
for $i=0,1,\ldots,k-1$. Furthermore, for $i=k$,
\begin{equation}\label{expk}
F_k(s,x(s,z,\e))=F_k\left(s,x(s,z,0)\right)+\CO(\e).
\end{equation}
From \eqref{x}, \eqref{exp3}, and \eqref{expk}, we get the following
equation:
\begin{equation}\label{exp4}
x(t,z,\e)=z+Q(t,z,\e) +
\sum_{i=0}^{k}\e^i\int_{0}^tF_i(s,x(s,z,0))ds+\CO(\e^{k+1}),
\end{equation}
where
\[
Q(t,z,\e)=\sum_{i=1}^{k-1}\e^{i} \sum_{l=1}^{i}\sum_{S_l}\int_{0}^t\dfrac{1}
{b_1!\,b_2!2!^{b_2}\cdots b_l!l!^{b_l}}\p^LF_{i-l}(s,x(s,z,0))
\bigodot_{j=1}^ly_j(s,z)^{b_j}ds.
\]

Finally, from \eqref{exp4}
\[
\begin{array}{RL}
x(t,z,\e)=&z+\int_{0}^tF_0(t,x(s,z,0))ds+\sum_{i=1}^{k-1}\e^{i}\Bigg(\int_{0}^t
F_i(s,x(s,z,0))\\
&+\sum_{l=1}^{i}\sum_{S_l}\dfrac{1}{b_1!\,b_2!2!^{b_2}\cdots
b_l!l!^{b_l}}\p^LF_{i-l}(s,x(s,z,0))\bigodot_{j=1}^ly_j(s,z)^{b_j}\,ds\Bigg)\\
&+\e^{k}\int_{0}^tF_k\left(s,x(s,z,0)\right)+\e^{k+1}
\int_{0}^tR(s,x(s,z,\e),\e)ds+\CO(\e^{k+1}).
\end{array}
\]

Now using this last expression of $x(t,z,\e)$ we conclude that
functions $y_i(t,z),$ defined in \eqref{yj} for $i=1,2,\ldots,k-1$,
can be computed recurrently from the following integral equation
\begin{equation}\label{ie}
\begin{array}{RL}
y_i(t,z)=&\left(\dfrac{\p^i x}{\p\e^i}(t,z,\e)\right)\Bigg|_{\e=0}\\
=&i!\int_{0}^t\Bigg(F_i(s,x(s,z,0)) +\sum_{l=1}^{i}\sum_{S_l}
\dfrac{1}{b_1!\,b_2!2!^{b_2}\cdots
b_l!l!^{b_l}}\\
&\cdot\p^LF_{i-l}(s,x(s,z,0))\bigodot_{j=1}^ly_j(s,z)^{b_j}\Bigg)ds\\
=&\int_{0}^t\left(A(s)y_i(s,z)+B_i(s)\right)ds,
\end{array}
\end{equation}
where
\[
\begin{array}{RL}
A(s)=&\p F_{0}(s,x(s,z,0)),\\
B_i(s)=&i!\Big(F_i(s,x(s,z,0))
+\sum_{S'_{i}}\dfrac{1}{b_1!\,b_2!2!^{b_2}\cdots
b_{i-1}!(i-1)!^{b_{i-1}}}\p^{I'}F_{0}(s,x(s,z,0))\\
&\bigodot_{j=1}^{i-1}y_j(s,z)^{b_j}+\sum_{l=1}^{i-1}\sum_{S_l}
\dfrac{1}{b_1!\,b_2!2!^{b_2}\cdots
b_l!l!^{b_l}}\p^LF_{i-l}(s,x(s,z,0))\\
&\bigodot_{j=1}^ly_j(s,z)^{b_j}\Big).
\end{array}\]

The integral equation \eqref{ie} is equivalent to the Cauchy problem
\[
\dfrac{\p}{\p t}y_i(t,z)=A(t)y_i(t,z)+B_i(t), \quad \text{with}
\quad y_i(0,z)=0,
\]
which has a unique solution given by
\[
\begin{array}{RL}
y_i(t,z)=&Y(t,z)\int_{0}^tY(s,z)^{-1}B_i(s)ds\\
=&i!\,Y(t,z)\int_{0}^tY(s,z)^{-1}\Bigg(F_i(s,x(s,z,0)) \\
&+\sum_{S'_{i}}\dfrac{1}{b_1!\,b_2!2!^{b_2}\cdots
b_i!i!^{b_i}}\p^{I'}F_{0}(s,x(s,z,0))\bigodot_{j=1}^{i-1}y_j(s,z)^{b_j}\\
&+\sum_{l=1}^{i-1}\sum_{S_l} \dfrac{1}{b_1!\,b_2!2!^{b_2}\cdots
b_l!l!^{b_l}}\p^LF_{i-l}(s,x(s,z,0))
\bigodot_{j=1}^ly_j(s,z)^{b_j}\Bigg)ds.
\end{array}
\]

Since
\[
x(t,z,0)=z+\int_{0}^tF_0(t,x(s,z,0))ds,
\]
we obtain
\[
x(t,z,\e)=x(t,z,0)+\sum_{i=1}^{k}\e^{i}\dfrac{y_i(t,z)}{i!}+\CO(\e^{k+1}).
\]
This concludes the proof of the lemma
\end{proof}

\begin{proof}[Proof of Theorem \ref{PSt1}]
Let $x(\cdot,z,\e):[0,t_{(z,\e)})\mapsto\mathbb{R}^n$  denote the
solution of system \eqref{s1} such that $x(0,z,\e)=z$. By Theorem
$8.3$ of \cite{Am} there exists a neighborhood $U$ of $z$ and $\e_1$
sufficiently small such that $t_{(z,\e)}>T$ for all $(z,\e)\in U\times
(-\e_1,\e_1)$. Let $h(z,\e):U\times (-\e_1,\e_1) \mapsto \mathbb{R}^n$ be the {\it displacement function}
 defined as
\begin{equation}\label{df}
h(z,\e)=x(T,z,\e)-z.
\end{equation}
Clearly $x(\cdot,\ov{z},\ov{\e})$,  for some $(\ov{z},\ov{\e})\in U\times
(-\e_1,\e_1)$, is a $T$-periodic solution of
system \eqref{s1} if and only if $h(\ov{z},\ov{\e})=0$. Studying the
zeros of \eqref{df} is equivalent to study the zeros of
\begin{equation}\label{zdf}
g(z,\e)=Y(T,z)^{-1}h(z,\e).
\end{equation}
From  Lemma \ref{l1} we have
\begin{equation}\label{aveq}
x(t,z,\e)=x(t,z,0)+\sum_{i=1}^{k}\e^{i}\dfrac{y_i(t,z)}{i!}+\CO(\e^{k+1}).
\end{equation}
 for all $(t,z)\in\s^1\times\D$, where $y_i$ is defined in
\eqref{smoothyi}. Hence substituting \eqref{aveq} into \eqref{zdf} it follows that
\begin{equation}\label{proofg}
g(z,\e)=\sum_{i=0}^{k} \e^i g_i(z)+\CO(\e^{k+1}),
\end{equation}
where $g_0(z)=Y^{-1}(t,z)\left(x(t,z,0)-z\right)$ and, for $i=1,2,\ldots,k,$ the function $g_i$ is defined in \eqref{smoothgi}.

From hypothesis (H$_b$) we know that $g_0(z)$ vanishes on the manifold $\CZ$, therefore hypothesis (H$_a$) holds for the function \eqref{proofg}. Moreover
\[
\begin{array}{rl}
\dfrac{\p g_0}{\p z}(z)=&Y(T,z)^{-1}\left(\dfrac{\p x}{\p z}(T,z,0)-Id\right)\vspace{0.2cm}\\
&=Y(T,z)^{-1}\left(Y(T,z)Y(0,z)^{-1}-Id\right)\vspace{0.2cm}\\
&=Y(0,z)^{-1}-Y(T,z)^{-1},
\end{array}
\]
which from hypothesis has its lower right corner $(n-m)\times (n-m)$
matrix  as being a nonsingular matrix $\Delta_{\al}$. Hence the result follows directly by applying Theorem \ref{LSt1}.
\end{proof}

\section{Birth of a limit cycle in a 3D polynomial system}\label{ApTB}

Consider the following 3D
autonomous polynomial differential system
\begin{align}\label{ex1}
\dot{u}=& -v+\e  \left(u^3-u^2-u v^2-\pi  v^3\right), \nonumber\\
\dot{v}=& u+ \e\left(\pi  u^3-1\right), \\
\dot{w}=&w-\e u. \nonumber
\end{align}
In the next proposition, as an application of Theorem \ref{PSt1}, we provide sufficient conditions for the
existence of an isolated periodic solution for the differential system \eqref{ex1}.
\begin{proposition}\label{pex1}
For $|\e|>0$ sufficiently small  system  \eqref{ex1} has an isolated
periodic solution $\varphi(t,\e)=\big(u(t,\e),\,v(t,\e),\,w(t,\e)
\big)$, such that
\begin{align}\label{Sex1}
u(t,\e)=&\sqrt{8 \,\e}\,\cos t +\CO(\e), \nonumber \\
v(t,\e)=&\sqrt{8 \,\e} \,\sin t +\CO(\e), \,\, \text{and}\vspace{0.2cm}\\
w(t,\e)=&\CO(\e).
\nonumber
\end{align}
\end{proposition}

We emphasize that the expression \eqref{Sex1} is not saying that the period of the solution $\varphi(t,\e)$ is $2\pi$. That is because we cannot assure the period of the order $\e$ functions.

\begin{proof}
Writing the differential system \eqref{ex1} in the cylindrical coordinates $(u,v,w)=(r \cos\theta,r
\sin\theta,w)$ we get
\begin{align}
\dot{r}=&\frac{ \e}{4}  \left(r^3+r^2 (r (\pi  \sin (4 \theta )+2 \cos
(2 \theta )+\cos (4 \theta ))-3 \cos \theta-\cos (3 \theta ))-4 \sin\theta
\right),\nonumber \\
\dot{\theta}=&1+\frac{\e}{4 r}  \left(r^2 (\sin\theta+\sin (3 \theta
)-r \sin (4 \theta )+\pi  r \cos (4 \theta )+3 \pi  r)-4 \cos
\theta\right), \nonumber \\
\dot{w}=&w-\e r  \cos \theta. \nonumber
\end{align}
Since $\dot \T\neq 0$ for $|\e|\neq0$ sufficiently small, we can take $\theta$ as
the new independent variable. So
\begin{equation}\label{nfex1}
\begin{array}{rl}
\dfrac{d r}{d\theta}=&\e F_{11}(\theta, z)+\e^2
F_{21}(\theta,z) +\CO_1(\e^3),\vspace{2mm}\\
\dfrac{d z}{d\theta}=&z+\e F_{12}(\theta, z) +\e^2
F_{22}(\theta, z)+\CO_2(\e^3),\\
\end{array}
\end{equation}
where $z=(r,w)\in \R^2$ and
\begin{align*}
F_{11}(\theta,z)=& \dfrac{1}{4} \big(r^3+r^2 (r (\pi
\sin (4 \theta )+2 \cos (2 \theta )+\cos (4 \theta ))-3 \cos \theta
-\cos (3 \theta ))\\
&-4 \sin \theta \big), \\
F_{12}(\theta,z)=&\frac{-1}{4} \big( 4 \cos \theta
\left(r^2-z\right)+r^2 z (\sin\theta +\sin (3 \theta )-r \sin (4
\theta )+\pi  r \cos (4 \theta )\\
&+3 \pi  r)\big), \\
F_{21}(\theta,z)=&\dfrac{-1}{16 r}\big(-4 \sin
\theta+r^3+r^2 (-3 \cos \theta-\cos (3 \theta )+r (\pi  \sin (4
\theta )\\
&+2 \cos (2 \theta )+\cos (4 \theta )))\big)\big(r^2 (\sin
\theta+\sin (3 \theta )-r \sin (4 \theta )+\pi  r \cos (4 \theta )\\
&+3 \pi  r)-4 \cos \theta \big),\\
F_{22}(\theta,z)=&\dfrac{1}{16 r^2}\big( r^2 (\sin
\theta+\sin (3 \theta )-r \sin (4 \theta )+\pi  r \cos (4 \theta )+3
\pi  r)-4 \cos \theta\big)\\
&\big( 4 \cos \theta \left(r^2-z\right)+r^2z(\sin \theta+\sin (3
\theta )-r \sin (4 \theta )+\pi  r \cos (4 \theta )\\ &+3 \pi
r\big).
\end{align*}

The differential system \eqref{nfex1} is $2\pi$-periodic in the variable $\T$ and it is written in the standard form form
\eqref{s1} with $F_0(\theta,z)=\big(0,z \big)$,
$F_1(\theta,z)=\big(F_{11}(\theta,z),
F_{12}(\theta,z)\big)$ and
$F_2(\theta,z)=\big(F_{21}(\theta,z),F_{22}
(\theta,z)\big)$. Moreover the solution of the unperturbed differential system \eqref{ups} for a initial condition $z_0=(r_0,w_0)$ is given by
\begin{equation*}
\Phi(\theta,\mbox{$z_0$})=\big(r_0,w_0 e^{\theta}\big).
\end{equation*}
Consider the set $\CZ\subset \R^2$ such that $\CZ=\lbrace (\alpha,0)
: \alpha >0\rbrace$. Clearly for each $z_\al \in \CZ$, the solution
$\Phi(\theta,z_\al)$ is $2 \pi$--periodic, and therefore the differential system \eqref{nfex1} satisfies hypothesis (H$_b$). Furthermore
the linear differential system \eqref{vs} corresponding to
\eqref{nfex1} has the following fundamental matrix solution
\begin{equation*}
Y(\theta,z)=\dfrac{\p\Phi}{\p z}(\T,z_0)=\begin{pmatrix}
1 & 0 \\
0 & e^\theta
\end{pmatrix},
\end{equation*}
which satifies $Y(0,z)=Id$. Now  in order to compute
the bifurcation functions  \eqref{fi} for the differential system \eqref{nfex1} we first obtain the
functions \eqref{smoothyi} corresponding to this system,
\begin{align}
y_0(\theta,z)&=Y(\theta,z)^{-1}\big(
0,\, (e^{\theta}-1)w\big),\nonumber\\
y_1(\theta,z)&=Y(\theta,z)^{-1}\bigg(\dfrac{r^2}{48}
\left(-36 \sin \theta-4 \sin (3 \theta )+6 \pi  r \sin ^2(2 \theta
)+3 r \sin (4 \theta )\right)\nonumber\\
&\dfrac{1}{48} \left(12\left(\theta  r^3-4\right)+24 \cos \theta
\left(r^3 \sin \theta+2\right)\right),\, \dfrac{r^2}{2}(\cos
\theta-\sin \theta)\nonumber\\
&-\dfrac{e^{\theta }r}{48} (w ((36 \pi  \theta -3)
r+16)+24)+\dfrac{e^{\theta }w}{48} \big(48 \sin \theta+r^2 (12 \cos
\theta \nonumber\\
&+4 \cos (3 \theta )-3 r (\pi  \sin (4 \theta )+\cos (4 \theta
))\big)\bigg), \nonumber\\
y_2(2\pi,z)&=Y(2\pi,z)^{-1}\bigg(
\dfrac{-\pi  r (3 r+4)}{4},\,\dfrac{e^{-2 \pi }}{40}\big(((3-2 \pi )
r-6) r^2+10 \big) \nonumber \\ &+\dfrac{1}{40}\big(r^2 ((\pi  (7+15
\pi )-3) r+6)-10\big)\bigg), \nonumber
\end{align}
and from  \eqref{smoothgi}
\begin{equation}\label{gex1}
g_i(z)=Y(2\pi,z)\dfrac{y_i(2\pi,z)}{i!}
\quad \text{for} \quad i=0,1,2.
\end{equation}
So the bifurcation functions \eqref{fi} corresponding to the
functions \eqref{gex1} become
\begin{equation}\label{bfexample1}
f_1(\al)= \dfrac{ \pi  \al^3
}{2},\quad f_2(\al)=\pi  \al (3 \al+4),\quad \text{and}\quad \CF^2(\al,\e)=\e
f_1(\al)+\e^2 f_2(\al).
\end{equation}

Now we must check  that the function \eqref{bfexample1} satisfies the
hypotheses for applying  Theorem \ref{PSt1}. So
$\det(\Delta_\al)=\left|D_w\pi^{\perp} g_0(z_\al)\right|=1-e^{-2 \pi }\neq 0$, and for $a_\e=\sqrt{9 \e^2+8 \e }+3 \e$ we have that
\[
\CF^2(a_\e,\e)=0\quad \text{and}\quad
|\p_\al\CF^2(a_\e,\e)|\geq \e^2\left(8-\big|9\e+3\sqrt{\e(8+9\e)}\big|\right).
\]
Thus it is easy to find $P_0>0$ satisfying
$|\p_\al\CF^2(a_\e,\e)|\geq \e^2P_0$. Hence, in terms of Theorem
\ref{PSt1}, we have $r=1$, $k=2$, $l=2,$ and $(k+r+1)/2=2=l$. So we can apply Theorem  \ref{PSt1} in order to prove the existence of an isolated periodic solution $\big(r(\T,\e),z(\T,\e)\big)$ of the differential system \eqref{nfex1} such that
\[
r(0,\e)=\sqrt{9 \e^2+8 \e }+3 \e+\CO(\e)=\sqrt{8\e}+\CO(\e)\quad\text{and}\quad w(0,\e)=\CO(\e).
\]
Since $\T(t)=t+\CO(\e)$, this proofs ends by going back through the cylindrical coordinate change of variables.
\end{proof}

\section{Averaged functions with a continuum of zeros}\label{bf}

One of the main difficulties in  applying the averaging method for
finding periodic solutions is to compute the zeros of the averaged function
associated to the differential system. In this section we are going to show
how Theorems \ref{LSt1} and \ref{PSt1} can be  combined in order to
deal with this problem. To be precise, consider the $T$-periodic
differential system $x'=F(t,x,\e)$ as defined in \eqref{s1}, with $F_0=0$.
Note that $Y(t,z)=Id$ for every $t\in\s^1$ and $z\in D$.

As shown in the proof of Theorem
\ref{PSt1}, $x(t,z,\e)$ is a $T$-periodic solution of \eqref{s1} if
and only if $z$ is a zero of the displacement function $h$, defined in \eqref{df}. In this case $h(z,\e)=g(z,\e)$, which reads
\begin{equation}\label{dpf}
h(z,\e)=x(T,z,\e)-z=\sum_{i=1}^k \e^i g_i(z)+O(\e^{k+1}),
\end{equation}
where the averaged functions $g_i(z),$ for $i=1,2,\ldots,k$, are defined in \eqref{smoothgi}. In order
to apply Theorem \ref{PSt1} we first compute
\begin{equation}\label{cfd}
\CF^k(\al,\e)=\sum^k_{i=1}\e^ig_i(\al),
\end{equation}
as defined in \eqref{cF}, and then we try to find $a_\e \in V $ such that
$\CF^k(a_\e,\e)=0$. After that, if all the hypotheses of Theorem \ref{PSt1} are fulfilled we obtain, from its proof, the existence of a branch of zeros $z(\al)$ of the displacement function \eqref{dpf}.

This task can be very complicate because there
is no general method to find $a_\e$.  Although  if there exist $r\in\{1,\ldots,k\}$,  an open subset
$\widetilde{V}\subset D$, and a smooth function
$\widetilde{\beta}:\cl(\widetilde{V})\rightarrow D$
such that $g_1=\ldots=g_{r-1}=0$, $g_r\neq0$, and
$g_r\big(\widetilde{\al},\widetilde{\beta}\big(\widetilde{ \al}
\big)\big)=0$ for all $\widetilde{\al}\subset \widetilde{V}$
then Theorem \ref{LSt1} may be used to reduce the dimension of
system \eqref{cfd}, helping then to find the solution $a_\e$. This strategy is a general method which generalizes the results obtained in \cite{CL}.  This procedure
will be illustrated in the next subsection.

\smallskip

\subsection{Maxwell-Bloch system}
In nonlinear optics, the Maxwell--Bloch equations are used to describe laser systems. For instance, in \cite{Ar}, these equations were obtained by coupling the Maxwell equations with the Bloch equation (a linear Schr\"{o}dinger like equation which describes the evolution of atoms resonantly coupled to the laser field).  Recently in \cite{LARZ}, it was identified weak foci and centers in the Maxwell-Bloch system, which can be written as
\begin{align}\label{wb}
\dot{u}=& -a u + v ,\nonumber \\
\dot{v}=&-b v+uw, \\
\dot{w}=&-c(w-\delta)-4u v. \nonumber
\end{align}
For $c=0$ the differential system \eqref{wb} has a singular line $\lbrace (u,v,w)|u=0,v=0\rbrace$; for $c\neq 0$ and  $ac(\delta-a b)\leq 0$ the differential system  \eqref{wb} has one equilibrium ${\bf p}_0=(0,0,\delta)$; and  for $c\neq 0$ and  $ac(\delta-a b)>0$ the differential system  \eqref{wb} has three equilibria ${\bf p}_\pm=\big(\pm u^*,\pm v^*,w^*\big)$ and ${\bf p}_0$ where
\begin{equation*}
u^*=\sqrt{\dfrac{c(\delta - a b )}{4 a}},\,  v^*=a\sqrt{\dfrac{c(\delta - a b )}{4 a}},\,  w^*=ab.
\end{equation*}
Using the above strategy we shall prove the following result:
\begin{proposition}\label{exT2}
Let $\omega \in (0,\, \infty)$, $(a,b,c)=\big(a_0-b_1\e + a_2\e^2,\, -a_0+b_1\e +b_2\e^2,\, c_1\e+c_2\e^2\big)$ and $\delta=-a_0^2-\omega^2$ with $a_0(a_2+b_2)> 0$, $c_1\neq 0$ and $\e$ a small parameter. Then for $|\e|\neq0$ sufficiently small the Maxwell-Bloch differential system \eqref{wb} has an isolated
periodic solution $\varphi(t,\e)=\big(u(t,\e),\,v(t,\e),\,w(t,\e)
\big)$ such that
\begin{align}\label{Sex11}
u(t,\e)=&\e\,\omega\sqrt{\dfrac{2(a_2+b_2)}{a_0}}\,\sin t +\CO(\e^2), \nonumber \\
v(t,\e)=&\e\,\omega\sqrt{\dfrac{2(a_2+b_2)}{a_0}}\big(a_0 \sin t+\omega\cos t\big) +\CO(\e^2), \,\, \text{and}\vspace{0.2cm}\\
w(t,\e)=&\delta-\e\,\dfrac{4\omega^2(a_2+b_2)}{c_1}+\CO(\e^2).
\nonumber
\end{align}
\end{proposition}
We emphasize again that the expression \eqref{Sex11} does not imply that the period of the solution $\varphi(t,\e)$ is $2\pi$. That is because we cannot assure the period of the order $\e^2$ functions.

\begin{proof}
Applying the change of variables $(u,v,w)=(\e V,\e (a_0V+\omega U),\delta+\e W)$,  the differential system \eqref{wb} reads
\begin{align}\label{syex22}
\dot{U}=&\,-\omega V+\dfrac{\e}{\omega}\big( VW -2a_0b_1V-b_1\omega U\big)+\e^2\Big(\dfrac{a_0(a_2-b_2)V}{\omega}-b_2U\Big),\nonumber\\
\dot{V}=&\,\omega U+\e b_1V-\e^2 a_2V,\\
\dot{W}=&\,\e\big(-c_1W-4V(a_0V+\omega U)\big)-\e^2c_ 2.\nonumber
\end{align}
In order to apply the strategy described above we must write the differential system \eqref{syex22} in the standard form \eqref{s1}. To this end we proceed as usual: first we consider the cylindrical change of variables $(U,V,W)=(r \cos \theta, r \sin \theta,w)$, where $r>0$; after checking that $\dot \T=\omega+\CO(\e)\neq0$, for $|\e|\neq0$ sufficiently small, we take $\theta$ as the new independent variable. Therefore the differential system \eqref{syex22} becomes equivalent to the non-autonomous differential system
\begin{equation}\label{Aex2}
\dfrac{d z}{d \theta}=\left(\dfrac{\dot r}{\dot \T},\dfrac{\dot w}{\dot \T}\right)=\e F_1(\theta ,z)+\e^2 F_2(\theta,z)+\CO(\e^3),
\end{equation}
where $z=(r,w)\in \R^+\times\R$ and $\T\in\mathbb{S}^1$. Moreover
\begin{equation}\label{F1F2}
\begin{array}{rl}
F_1(\theta , z)=\!\!\!\!&\Bigg(\dfrac{r}{2\omega^2}\big((w-2 a_0b_1)\sin(2\theta)-2b_1\omega \cos(2\theta) \big),\\
&-\dfrac{\big(c_1 w+4r^2 \sin\theta (\omega \cos \theta + a_0 \sin\theta)\big)}{\omega} \Bigg),\\
F_2(\theta , z)=\!\!\!\!&\Bigg( \dfrac{1}{2\omega^4}\big(2b_1\omega\cos\theta+(2a_0b_1-w)\sin\theta \big)\big(2b_1\omega\cos(2\theta)+(2a_0b_1\\
&-w)\sin(2\theta)\big)+r\omega^2\big( (a_2-b_2)(\omega\cos(2\theta)+a_0\sin(2\theta))-(a_2+b_2) \big),\\
& \dfrac{\big(2b_1\omega\cos\theta+(2a_0b_1-w)\sin\theta\big)}{\omega^2}\big(c_1w+4r^2\sin\theta(\omega\cos\theta+a_0\sin\theta) \big)\Bigg).\\
\end{array}
\end{equation}
Now the prime denotes the derivative with respect to the variable $\T$.

\smallskip

For the differential system \eqref{Aex2} we have that $ F_0(\theta, z)=0$. Then $x(\theta, z,0)=(r,w)$ is the solution to the unperturbed  system  and $Y(t,z)=Id$ is its corresponding fundamental matrix.  In this case the averaged functions reads
\begin{equation}\label{afexample2}
\begin{array}{rl}
g_1(z)=&\left(0, -\dfrac{2\pi\big(2a_0r^2+c_1w\big)}{\omega} \right), \\
g_2(z)=&\Bigg(\dfrac{\pi r\big(3a_0r^2+c_1w-2(a_2+b_2)\omega^2\big)}{2\omega^3},\dfrac{\pi}{\omega^3}\big((2a_0b_1-w)(6a_0r^2+c_1w)\\
&+2c_1\pi(2a_0r^2+c_1w)\omega+2\big((2b_1+c_1)r^2-c_2w\big)\omega^2\big)\Bigg).
\end{array}
\end{equation}

From here instead of following the steps of Theorem \ref{PSt1} we are going to use Theorem \ref{LSt1} to find directly a branch of zeros of the displacement function \eqref{dpf}.  To do this we define the function $\widetilde{g}(z,\e)=h(z,\e)/\e$, where now
$\widetilde{g}(z,\e)= \widetilde{g}_0(z)
+\e\widetilde{g}_1(z)
+\CO(\e^2)$ with
$\widetilde{g}_0(z)=g_1(z)$ and
$\widetilde{g}_1(z)=g_2(z)$.
Note that the averaged function $\widetilde{g}_0(z)=g_1(z)$ vanishes on the manifold
$$
\widetilde{\CZ}=\Bigg\{z_{\al}=\Bigg(
\al,-\dfrac{2 a_0 \al^2}{c_1}\Bigg):\,\al>0 \Bigg\}.
$$
Furthermore,  $\Delta_{\al}=-(2\pi c_1)/\omega$ is the lower right corner of the
Jacobian matrix $D\widetilde{g} _ 0(z_{\al})$ for all
$z_{\al}\in \widetilde{\CZ}$. Computing then the bifurcation
function \eqref{fi} corresponding to $\widetilde{g}(z,\e)$ we get
\begin{equation*}
\widetilde{f}_1(\al)=\dfrac{\pi \al \big(a_0\al^2-2(a_2+b_2)\omega^2 \big)}{2\omega^3},
\end{equation*}
and
$\widetilde{\CF}^1(\al,\e)=\e\widetilde{f}_1(\al)$. Solving the equation $\widetilde{\CF}^1(\al,\e)=0$ we find
$$a_\e=\al_0=\omega\sqrt{\dfrac{2(a_2+b_2)}{a_0}}.$$
Moreover, $\left|
\p_\al\widetilde{\CF}^1(\al_0,\e)\right|=2\e \pi(a_2+b_2)/\omega$ so it is clear that hypotheses $(iii)$ and $(iv)$ of Theorem  \ref{LSt1}  are fulfilled with $l=1$, $r=1$ and $k=1$. Thus, for $|\e|\neq 0$ sufficiently small, it follows that there exists
\begin{equation}\label{ap1z}
z(\e)=\left(\omega\sqrt{\dfrac{2(a_2+b_2)}{a_0}},-\dfrac{4\omega^2(a_2+b_2)}{c_1}\right)+\CO(\e)
\end{equation}
 such that $\widetilde{g}(z(\e),\e)=h(z(\e),\e)/\e=0$ for every $|\e|\neq0$ sufficiently small. Therefore we conclude that there exists a $2\pi$-periodic solution periodic $(r(\T,\e),w(\T,\e))$ of the non-autonomous differential system \eqref{Aex2} satisfying $(r(\T,0),w(\T,0))=z(0)$. Since $\T(t)=\omega t+\CO(\e)$, this proofs ends by going back through the cylindrical coordinate change of variables and then doing $(u,v,z)=\e(V,a_0V+\omega U,W)$ .
\end{proof}

\subsection{Stability}
We have seen that  the averaged functions \eqref{afexample2} up to order $2$ were sufficient for detecting the existence of a periodic solution of the differential system \eqref{wb}.  Now we show that the higher order averaged functions may play an important role for studying the stability of the periodic solution $\varphi(t,\e)$ provided by Theorem \ref{PSt1}.

Clearly the stability of the periodic solution $\varphi(t,\e)$ can be derived from the eigenvalues of the Jacobian matrix of the displacement function $D_zh(z(\e),\e)$  evaluated at $z(\e)=\varphi(0,\e)$ .  From equation \eqref{ap1z} we can write $z(\e)=z_0+\CO(\e^2)$. Moreover, since in this case $Y(t,z)=Id$ then  $D_zh(z(\e),\e)=\e Dg_1(z_0)+\CO(\e),$ where
\begin{equation*}
Dg_1(z_0)=\begin{pmatrix} 0 & 0\\
-8\pi\sqrt{2 a_0(a_2+b_2)}& -\dfrac{2\pi c_1}{\omega}
\end{pmatrix}.
\end{equation*}
So a first approximation of the eigenvalues $\lambda_{\pm}$ of the Jacobian matrix  $D_zh(z(\e),\e)$ is given by
\begin{equation}\label{app1}
\lambda_+=\CO(\e^2),\quad  \lambda_-=-\e\dfrac{2 \pi c_1}{\omega}+\CO(\e^2).
\end{equation}
Clearly the stability of the periodic solution $\varphi(t,\e)$ cannot be completely described by these expressions. Now we show how the higher order bifurcation functions and averaging functions can be used to better analyses the stability of the periodic solution.

We recall that, after some changes of coordinates, the differential system \eqref{wb} can be transformed into the standard form \eqref{Aex2}. Expanding it in power series of $\e$ up to order $3$, the differential system  \eqref{Aex2}  becomes
\begin{equation*}\label{Aex2.3}
\dfrac{d z}{d\T}=\e F_1(\theta , z)+\e^2F_2(\theta,z)+\e^3F_3(\theta,z)+\CO(\e^4),
\end{equation*}
where $ F_1$ and $ F_2$ are given in \eqref{F1F2} and
\begin{align*}
 F_3(\theta,z)=&\Bigg(\dfrac{\pi r}{4\omega^5}\Big(-3(a_0b_1-w)\big(5a_0r^2+c_1w\big)-2c_1\pi(2a_0r^2+c_1w)\omega+\big(4a_0b_1(a_2\\
&+b_2)-3(2b_1+c_1)r^2-2(a_2+b_2-c_2)w\big)\omega^2\Big),\\
&\dfrac{\pi}{12\omega^5}\Big(12 \pi  \omega  \left(a_0^2 \left(6 r^4-16 b_1 c_1 r^2\right)+2 a_0 c_1 w \left(7 r^2-2 b_1 c_1\right)+3 c_1^2 w^2\right)\\
&-2 \omega ^2 \left(w \left(6 a_0 (a_2 c_1-2 b_1 c_2-b_2 c_1)+6 b_1^2 c_1-9 r^2 (4 b_1+3 c_1)+8 \pi ^2 c_1^3\right)\right.\\
&\left.+a_0 r^2 \left(36 a_0 (a_2-b_2)+108 b_1^2+36 b_1 c_1+2 \left(8 \pi ^2-3\right) c_1^2-45 r^2\right)+6 c_2 w^2\right)\\
&+24 \pi  \omega ^3 \left(r^2 (2 a_0 (a_2+b_2+c_2)-c_1 (2 b_1+c_1))+2 c_1 c_2 w\right)\\
&-9 (w-2 a_0 b_1)^2 \left(10 a_0 r^2+c_1 w\right)+24 r^2 \omega ^4 (c_2-2 a_2)\Bigg).
\end{align*}
From \eqref{smoothyi} and \eqref{fi} we compute the third averaged function and the second bifurcation function, respectively, as
\begin{align*}
g_3(z)=&\Bigg(\dfrac{\pi  r }{4 \omega ^5}\left(\omega ^2 \left(4 a_0 b_1 (a_2+b_2)-2 z (a_2+b_2-c_2)-3 r^2 (2 b_1+c_1)\right)\right.\\
&\left.-3 (2 a_0 b_1-z) \left(5 a_0 r^2+c_1 z\right)-2 \pi  c_1 \omega  \left(2 a_0 r^2+c_1 z\right)\right),\\
&\dfrac{\pi }{12 \omega ^5} \left(12 \pi  \omega  \left(a_0^2 \left(6 r^4-16 b_1 c_1 r^2\right)+2 a_0 c_1 z \left(7 r^2-2 b_1 c_1\right)+3 c_1^2 z^2\right)\right.\\
&\left.-2 \omega ^2 \left(z \left(6 a_0 (a_2 c_1-2 b_1 c_2-b_2 c_1)+6 b_1^2 c_1-9 r^2 (4 b_1+3 c_1)+8 \pi ^2 c_1^3\right)\right.\right.\\
&\left.\left.+a_0 r^2 \left(36 a_0 (a_2-b_2)+108 b_1^2+36 b_1 c_1+2 \left(8 \pi ^2-3\right) c_1^2-45 r^2\right)\right.\right.\\
&\left.\left.+6 c_2 z^2\right)+24 \pi  \omega ^3 \left(r^2 (2 a_0 (a_2+b_2+c_2)-c_1 (2 b_1+c_1))+2 c_1 c_2 z\right)\right.\\
&\left.-9 (z-2 a_0 b_1)^2 \left(10 a_0 r^2+c_1 z\right)+24 r^2 \omega ^4 (c_2-2 a_2)\right) \Bigg)
\end{align*}
and
$$
\widetilde{f}_2(\al)=-\frac{\pi  r \left(10 a_0^2 r^2 \left(b_1 c_1+r^2\right)+\omega ^2 \left(c_1 r^2 (2 b_1+c_1)-4 a_0 (a_2+b_2) \left(b_1 c_1+r^2\right)\right)\right)}{4 c_1 \omega ^5}.
$$
So $\widetilde\CF^2(\alpha,\e)=\e \widetilde{f}_1(\al)+\e^2 \widetilde{f}_2(\al)$.
As shown in the previous subsection $a_\e=\al_0$ is a simple root of the function $\widetilde{f}_1(\al)$. Using the Implicit Function Theorem we find a branch of zeros of the equation $\CF^2(\al,\e)=0$  having the form $\al=\ov{a}_\e=\al_0+\e\al_1+\CO(\e^2)$, where
\[
\al_1=\sqrt{\dfrac{a_2+b_2}{2a_0}}\left(\dfrac{8 a_0^2 b_1 c_1+\omega ^2 (16 a_0 (a_2+b_2)+c_1 (2 b_1+c_1))}{2 |a_0| c_1 \omega }\right).
\]
Note that $\ov{a}_\e$ satisfies the hypotheses $(iii)$ and $(iv)$ of Theorem \ref{LSt1} for $r=1$, $l=1$ and $k=2$. Using the relation  $|\pi z(\e)-\pi z_{\ov{a}_\e}|=| \al(\e)-\ov{a}_\e|=\CO\big(\e^{2}\big)$, provided by Theorem \ref{LSt1}, we write $\al(\e)=\al_0+\e\al_1+\CO(\e^2)$. From Claim \ref{c1} of the proof of Theorem \ref{LSt1} we get
\begin{align*}
\ov{\beta}(\al(\e),\e)=&\beta\big(\al(\e)\big)+\e\ga_1\big(\al(\e)\big)+\CO(\e^2)\\
=&\beta\big(\al_0+\e\al_1+\CO(\e^2)\big)+\e\ga_1\big(\al_0+\e\al_1+\CO(\e^2)\big)+\CO(\e^2).
\end{align*}
Expanding $\ov{\beta}(\al(\e),\e)$ in powers series of $\e$ we have $ \ov{\beta}(\al(\e),\e)=\beta_0+\e\beta_1+\CO(\e^2)$ where
\begin{align*}
\beta_0=&\dfrac{(a_2+b_2)\omega^2}{c_1},\\
\beta_1=&\dfrac{4(a_2+b_2)\big(6a_0^2b_1c_1+\big(16a_0(a_2+b_2)+c_1(2b_1+c_1)\big)\omega^2\big)}{a_0c_1^2}.
\end{align*}
Finally we obtain $z(\e)=\big(\al(\e),\ov{\beta}\big(\al(\e),\e\big)\big)=z_0+\e z_1+\CO(\e^2)$, with $z_0=(\al_0,\beta_0)$ and $z_1=\big(\al_1,\beta_1\big)$. Then we compute the Jacobian matrix of the displacement function \eqref{df} evaluated at $z(\e)$  as
\begin{align*}
D_zh(z(\e),\e)&=\e D_zg_1(z(\e))+\e^2D_zg_2(z(\e))+\CO(\e^3)\nonumber\\
&=\e D_zg_1\big(z_0+\e z_1+\CO(\e^2)\big)+\e^2D_zg_2\big(z_0+\e z_1+\CO(\e^2)\big)+\CO(\e^3)\nonumber\\
&=\e D_zg_1(z_0)+\e^2\left(D^2_zg_1(z_0)z_1+D_zg_2(z_0)\right)+\CO(\e^3).
\end{align*}
Let $D_zg_1(z_0)=\big(p_{ij}\big)_{2\times 2}$ and $D_zg_2(z_0)=\big(q_{ij}\big)_{2\times 2}$ then expanding $D_zh(z(\e),\e)$ in Taylor series around $\e=0$ we have  $D_zh(z(\e),\e)= \e A_1+\e^2 A_2+\CO(\e^3)$ with $A_1=D_zg_1(z_0)$ and $A_2=\big(D_zp_{ij}(z_0).z_1+q_{ij}(z_0)\big)_{2\times 2}$. Therefore we may improve the approximation \eqref{app1} of the eigenvalues $\lambda_{\pm}$ of $D_zh(z(\e),\e)$ as
\begin{align*}
\lambda_+=&\e^2\dfrac{2\pi(a_2+b_2)}{\omega}+\CO(\e^3),\\
\lambda_-=&-\e\dfrac{2c_1\pi}{\omega}+\e^2\dfrac{2\pi\big(a_0b_1c_1+\omega\big(c_1^2\pi-c_2\omega\big)\big)}{\omega^3}+\CO(\e^3).
\end{align*}
Hence we can deduce the following statements about the stability of the periodic solution $\varphi(t,\e)=x(t,z(\e),\e)$. Recall that from, hypotheses of Proposition \ref{exT2}, $a_0(a_2+b_2)>0$. So:
\begin{itemize}
\item[(a)] If $\e c_1<0$ the solution $\varphi(t,\e)$ has at least one  unstable direction.
\item[(b)] If $a_2+b_2>0$   and $a_0>0$ then the solution $\varphi(t,\e)$ has at least one  unstable direction.
\item[(c)] If $a_2+b_2<0$, $\e c_1>0$ and  $a_0<0$ then the solution $\varphi(t,\e)$ is asymptotically stable.
\end{itemize}

The following figures illustrate the behavior of the Maxwell--Block system \eqref{wb} satisfying the hypotheses of Proposition \ref{exT2} with $a_0=-1$, $a_2=-2$, $b_1=1$, $b_2=-2$, $c_1=2$, $c_2=1$, $\omega=1$ and $\e=1/25$.
\begin{figure}[!htbp]
  \centering
  \subfloat[ ]{\includegraphics[width=0.5\textwidth]{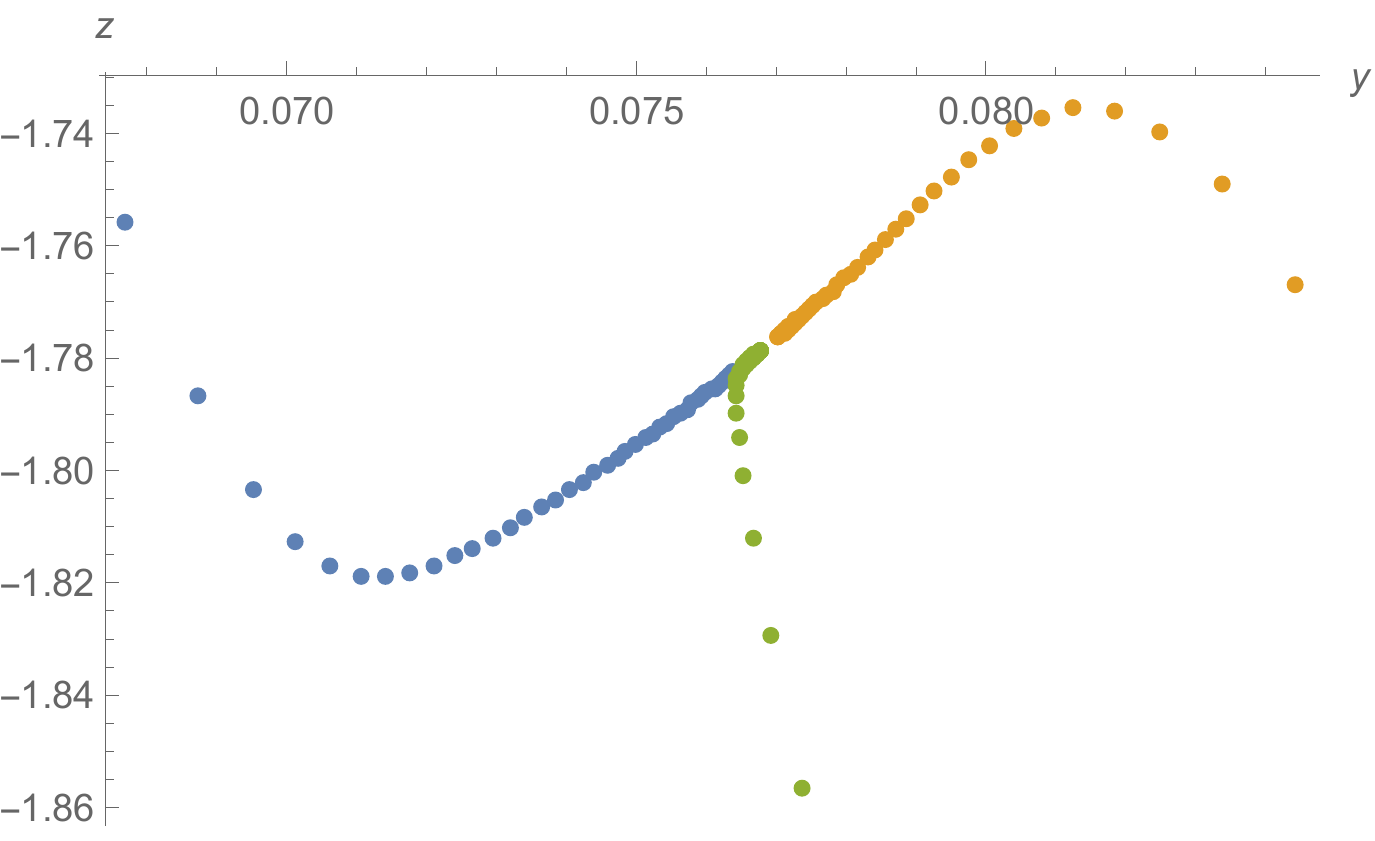}\label{fig:f1}}
  \hfill
  \subfloat[ ]{\includegraphics[width=0.5\textwidth]{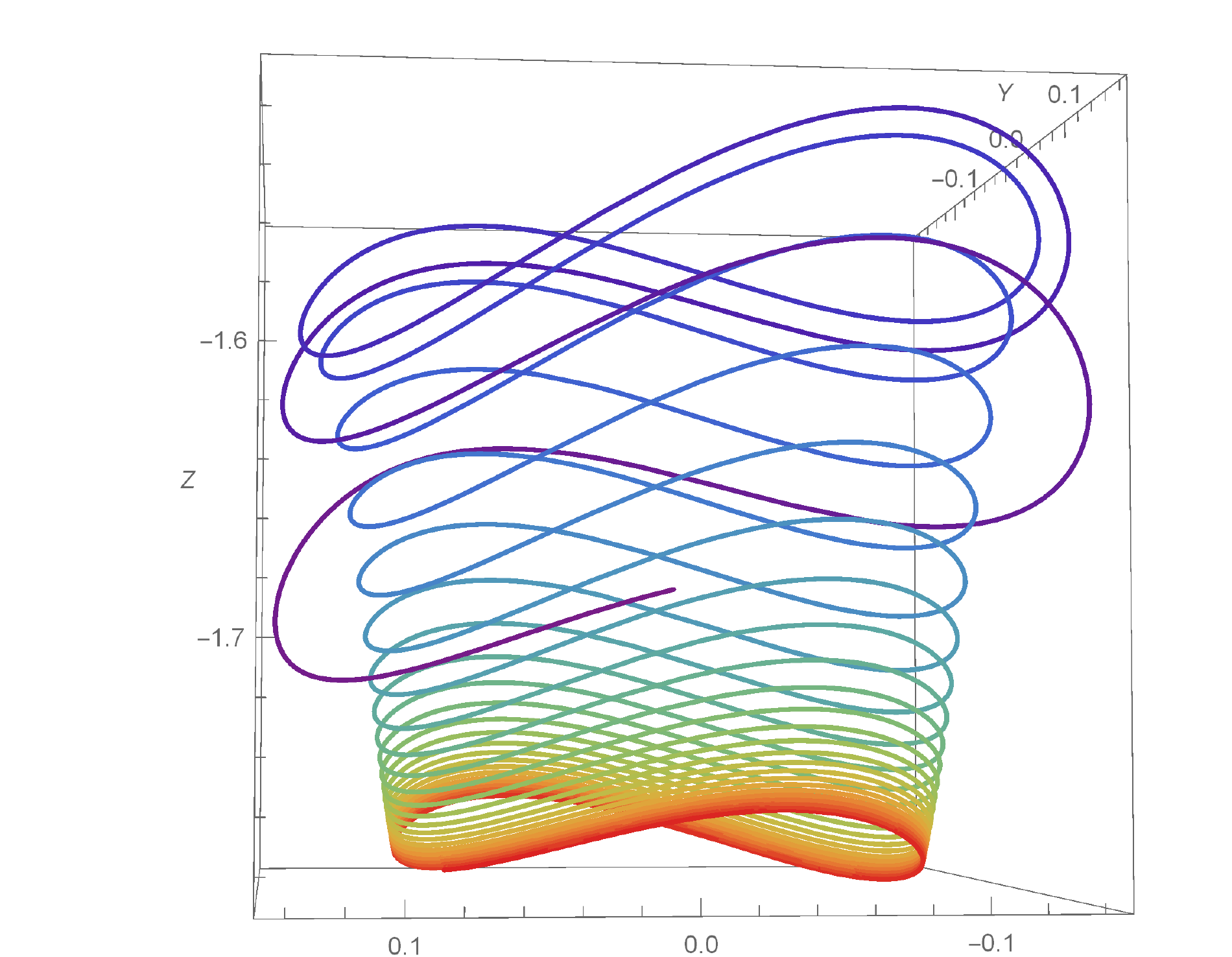}\label{fig:f2}}
  \caption{\Small (A) Transversal section with $u=0$ and $v>0$. (B) Solution starting at $(0,\e \omega^2(2(a_2+b_2)/a_0)^{1/2},\delta-4\e\omega^2(a_2+b_2)/c_1)$ being attracted by the limit cycle \eqref{Sex11}.}
\end{figure}

\section*{Appendix A: Bifurcation functions up to order $5$.}

In this appendix we develop the recurrences given by Theorems \ref{LSt1} and \ref{PSt1} to compute explicitly the expressions of the bifurcation functions and the averaged functions up to order $5$. As far as we know we are the first to provide these expressions.

\smallskip

From the recurrences \eqref{fi} and \eqref{y}, we explicitly develop the expressions of the bifurcation functions $f_i: V \rightarrow \R^m$, for $i=1,2,..,5$, as stated in Theorem \ref{LSt1}. Recall that $\G_{\al}=(\p\pi g_0/\p b)(z_{\al})$. So
\begin{align*}
f_1(\al)=&\G_\al \ga_ 1(\al)+\pi g_1(z_\al), \vspace{0.3cm}\\
\ga_1(\al)=&-\De_\al^{-1}\pi^\perp g_1(z_\al), \vspace{0.3cm}\\
f_2(\al)=&\dfrac{1}{2}\G_\al\ga_2(\al)+\dfrac{1}{2}\dfrac{\p^2\pi
g_0}{\p b^2}(z_\al)\ga_1(\al)^2+\dfrac{\p\pi g_1}{\p
b}(z_\al)\ga_1(\al)+\pi g_2(z_\al), \vspace{0.3cm}\\
\ga_2(\al)=&-\De_\al^{-1}\Bigg(\dfrac{\p^2\pi^\perp g_0}{\p
b^2}(z_\al)\ga_1(\al)^2 +2\dfrac{\p\pi^\perp g_1}{\p
b}(z_\al)\ga_1(\al)+2\pi^\perp g_2(\al) \Bigg),\vspace{0.3cm}\\
f_3(\al)=&\dfrac{1}{6}\G_\al\ga_3(\al)+\dfrac{1}{6}\dfrac{\p^3\pi
g_0}{\p b^3}(z_\al)\ga_1(\al)^3+\dfrac{1}{2}\dfrac{\p^2\pi g_0}{\p
b^2}(z_\al)\ga_1(\al)\odot \ga_2(\al)\vspace{0.3cm}\\
&+\dfrac{1}{2}\dfrac{\p^2\pi g_1}{\p
b^2}(z_\al)\ga_1(\al)^2+\dfrac{1}{2}\dfrac{\p\pi g_1}{\p
b}(z_\al)\ga_2(\al)+\dfrac{\p\pi g_2}{\p b}(z_\al)\ga_1(\al)\vspace{0.3cm}\\
&+\pi g_3(z_\al), \vspace{0.3cm}\\
\ga_3(\al)=&-\De_\al^{-1}\Bigg(\dfrac{\p^3\pi^\perp g_0}{\p
b^3}(z_\al)\ga_1(\al)^3 +3\dfrac{\p^2\pi^\perp g_0}{\p
b^2}(z_\al)\ga_1(\al)\odot \ga_2(\al) \vspace{0.3cm}\\
&+3\dfrac{\p^2\pi^\perp g_1}{\p
b^2}(z_\al)\ga_1(\al)^2+2\dfrac{\p\pi^\perp g_1}{\p
b}(z_\al)\ga_2(\al)+6\dfrac{\p\pi^\perp g_2}{\p
b}(z_\al)\ga_1(\al)\vspace{0.3cm}\\
&+6\pi^\perp g_3(\al) \Bigg),\vspace{0.3cm}\\
f_4(\al)=&\dfrac{1}{24}\G_\al \ga_4(\al)+\dfrac{1}{24}\dfrac{\p^4\pi
g_0}{\p b^4}(z_\al)\ga_1(\al)^4+\dfrac{1}{4}\dfrac{\p^3\pi g_0}{\p
b^3}(z_\al)\ga_1(\al)^2\odot\ga_2(\al)\vspace{0.3cm}\\
&+\dfrac{1}{8}\dfrac{\p^2\pi g_0}{\p
b^2}(z_\al)\ga_2(\al)^2+\dfrac{1}{6}\dfrac{\p^2\pi g_0}{\p
b^2}(z_\al)\ga_1(\al)\odot\ga_3(\al)\vspace{0.3cm}\\
&+\dfrac{1}{6}\dfrac{\p^3\pi g_1}{\p
b^3}(z_\al)\ga_1(\al)^3+\dfrac{1}{2}\dfrac{\p^2\pi g_1}{\p
b^2}(z_\al)\ga_1(\al)\odot\ga_2(\al)+\dfrac{1}{6}\dfrac{\p\pi
g_1}{\p b}(z_\al)\ga_3(\al)\vspace{0.3cm}\\
&+\dfrac{1}{2}\dfrac{\p^2\pi g_2}{\p
b^2}(z_\al)\ga_1(\al)^2+\dfrac{1}{2}\dfrac{\p\pi g_2}{\p
b}(z_\al)\ga_2(\al)+\dfrac{\p\pi g_3}{\p b}(z_\al)\ga_1(\al)+\pi g_4(z_\al),\vspace{0.3cm}\\
\ga_4(\al)=&-\De_\al^{-1}\Bigg(\dfrac{\p^4\pi^\perp g_0}{\p
b^4}(z_\al)\ga_1(\al)^4 +3\dfrac{\p^2\pi^\perp g_0}{\p
b^2}(z_\al)\ga_2(\al)^2+4\dfrac{\p^2\pi^\perp g_0}{\p
b^2}(z_\al)\ga_1(\al)\odot\ga_3(\al)\vspace{0.3cm}\\
&+6\dfrac{\p^3\pi^\perp g_0}{\p
b^3}(z_\al)\ga_1(\al)^2\odot\ga_2(\al)+4\dfrac{\p\pi^\perp g_1}{\p
b}(z_\al)\ga_3(\al)+12\dfrac{\p^2\pi^\perp g_1}{\p
b^2}(z_\al)\ga_1(\al)\odot\ga_2(\al)\vspace{0.3cm}\\
&+4\dfrac{\p^3\pi^\perp g_1}{\p
b^3}(z_\al)\ga_1(\al)^3+12\dfrac{\p\pi^\perp g_2}{\p
b}(z_\al)\ga_2(\al)+12\dfrac{\p^2\pi^\perp g_2}{\p
b^2}(z_\al)\ga_1(\al)^2\vspace{0.3cm}\\ &+24\dfrac{\p\pi^\perp g_3}{\p
b}(z_\al)\ga_1(\al)\Bigg),\vspace{0.3cm}\\
f_5(\al)=&\dfrac{1}{120}\G_\al
\ga_5(\al)+\dfrac{1}{12}\dfrac{\p^2\pi g_0}{\p
b^2}(z_\al)\ga_2(\al)\odot\ga_1(\al)+\dfrac{1}{24}\dfrac{\p^2\pi
g_0}{\p b^2}(z_\al)\ga_1(\al)\odot\ga_4(\al)\vspace{0.3cm}\\
&+\dfrac{1}{8}\dfrac{\p^3\pi g_0}{\p
b^3}(z_\al)\ga_1(\al)\odot\ga_2(\al)^2+\dfrac{1}{12}\dfrac{\p^3\pi
g_0}{\p b^3}(z_\al)\ga_1(\al)^2\odot\ga_3(\al)\vspace{0.3cm}\\
&+\dfrac{1}{12}\dfrac{\p^4\pi g_0}{\p
b^4}(z_\al)\ga_1(\al)^3\odot\ga_2(\al)+\dfrac{1}{120}\dfrac{\p^5\pi
g_0}{\p b^5}(z_\al)\ga_1(\al)^5\vspace{0.3cm}\\
&+\dfrac{1}{24}\dfrac{\p\pi g_1}{\p
b}(z_\al)\ga_4(\al)+\dfrac{1}{8}\dfrac{\p^2\pi g_1}{\p
b^2}(z_\al)\ga_2(\al)^2+\dfrac{1}{6}\dfrac{\p^2\pi g_1}{\p
b^2}(z_\al)\ga_1(\al)\odot\ga_3(\al)\vspace{0.3cm}\\
&+\dfrac{1}{4}\dfrac{\p^3\pi g_1}{\p
b^3}(z_\al)\ga_1(\al)^2\odot\ga_2(\al)+\dfrac{1}{24}\dfrac{\p^4\pi
g_1}{\p b^4}(z_\al)\ga_1(\al)^4+\dfrac{1}{6}\dfrac{\p\pi g_2}{\p
b}(z_\al)\ga_3(\al)\vspace{0.3cm}\\
&+\dfrac{1}{2}\dfrac{\p^2\pi g_2}{\p
b^2}(z_\al)\ga_1(\al)\odot\ga_2(\al)+\pi
g_4(z_\al)+\dfrac{1}{6}\dfrac{\p^3\pi g_2}{\p
b^3}(z_\al)\ga_1(\al)^3\vspace{0.3cm}\\
&+\dfrac{1}{2}\dfrac{\p\pi g_3}{\p
b}(z_\al)\ga_2(\al)+\dfrac{1}{2}\dfrac{\p^2\pi g_3}{\p
b^2}(z_\al)\ga_1(\al)^2+\dfrac{\p\pi g_4}{\p b}(z_\al)\ga_1(\al)\vspace{0.3cm}\\
&+\pi g_5(z_\al),\vspace{0.3cm}\\
\ga_5(\al)=&-\De_\al^{-1}\Bigg(10\dfrac{\p^2\pi^\perp g_0}{\p
b^2}(z_\al)\ga_2(\al)\odot\ga_3(\al) +5\dfrac{\p^2\pi^\perp g_0}{\p
b^2}(z_\al)\ga_1(\al)\odot\ga_4(\al)\vspace{0.3cm}\\
&+15\dfrac{\p^3\pi^\perp g_0}{\p
b^3}(z_\al)\ga_1(\al)\odot\ga_2(\al)^2+10\dfrac{\p^3\pi^\perp
g_0}{\p b^3}(z_\al)\ga_1(\al)^2\odot\ga_3(\al)\vspace{0.3cm}\\
&+10\dfrac{\p^4\pi^\perp g_0}{\p b^4}(z_\al)\ga_1(\al)^3\odot
\ga_2(\al)+\dfrac{\p^5\pi^\perp g_0}{\p b^5}(z_\al)\ga_1(\al)^5\vspace{0.3cm}\\
&+5\dfrac{\p\pi^\perp g_1}{\p
b}(z_\al)\ga_4(\al)+15\dfrac{\p^2\pi^\perp g_1}{\p
b^2}(z_\al)\ga_2(\al)^2+20\dfrac{\p^2\pi^\perp g_1}{\p
b^2}(z_\al)\ga_1(\al)\odot\ga_3(\al) \vspace{0.3cm}\\
&+30\dfrac{\p^3\pi^\perp g_1}{\p
b^2}(z_\al)\ga_1(\al)^2\odot\ga_2(\al)+5\dfrac{\p^4\pi^\perp g_1}{\p
b^4}(z_\al)\ga_1(\al)^4\vspace{0.3cm}\\
&+20\dfrac{\p\pi^\perp g_2}{\p
b}(z_\al)\ga_3(\al)+60\dfrac{\p^2\pi^\perp g_2}{\p
b^2}(z_\al)\ga_1(\al)\odot\ga_2(\al)\vspace{0.3cm}\\
&+20\dfrac{\p^3\pi^\perp g_2}{\p
b^3}(z_\al)\ga_1(\al)^3+60\dfrac{\p\pi^\perp g_3}{\p
b}(z_\al)\ga_2(\al)\vspace{0.3cm}\\ &+60\dfrac{\p^2\pi^\perp g_3}{\p
b^2}(z_\al)\ga_1(\al)^2+120\dfrac{\p\pi^\perp g_4}{\p
b}(z_\al)\ga_1(\al)\Bigg)
\end{align*}

The averaged functions, as
stated in Theorem \ref{PSt1}, are computed as follows:
\[
g_i(z)=Y(T,z)^{-1}\dfrac{y_i(T,z)}{i!}.
\]
So, from the recurrence \eqref{smoothyi}, we explicitly develop the expressions of $y_i$, for $i=0,1,\dots,5$.
\begin{align*}
y_0(t,z)=&x(t,z,0)-z, \vspace{0.3cm}\\
y_1(t,z)=&Y(t,z)\int_0^t Y(\tau,z)^{-1}F_1(\tau,x(\tau,z,0))\mathrm{d}\tau,\vspace{0.3cm}\\
y_2(t,z)=&Y(t,z)\int_0^t
Y(\tau,z)^{-1}\Bigg[2F_2(\tau,x(\tau,z,0))+2\dfrac{\p F_1}{\p
x}(\tau,x(\tau,x,0))y_1(\tau,z) \vspace{0.3cm}\\
&+\dfrac{\p^2 F_0}{\p x^2}(\tau,x(\tau,z,0))y_1(\tau,z)^2 \Bigg]\mathrm{d}\tau,\vspace{0.3cm}\\
y_3(t,z)=&Y(t,z)\int_0^t
Y(\tau,z)^{-1}\Bigg[6F_3(\tau,x(\tau,z,0))+6\dfrac{\p F_2}{\p
x}(\tau,x(\tau,x,0))y_1(\tau,z) \vspace{0.3cm}\\
&+3\dfrac{\p^2 F_1}{\p
x^2}(\tau,x(\tau,z,0))y_1(\tau,z)^2+3\dfrac{\p F_1}{\p
x}(\tau,x(\tau,z,0))y_2(\tau,z)\vspace{0.3cm}\\
&+3\dfrac{\p^2 F_0}{\p x^2}(\tau,x(\tau,z,0))y_1(\tau,z)\odot
y_2(\tau,z)+\dfrac{\p^3 F_0}{\p
x^3}(\tau,x(\tau,z,0))y_1(\tau,z)^3\Bigg]\mathrm{d}\tau,\vspace{0.3cm}\\
y_4(t,z)=&Y(t,z)\int_0^t
Y(\tau,z)^{-1}\Bigg[24F_4(\tau,x(\tau,z,0))+24\dfrac{\p F_3}{\p
x}(\tau,x(\tau,x,0))y_1(\tau,z) \vspace{0.3cm}\\
&+12\dfrac{\p^2 F_2}{\p
x^2}(\tau,x(\tau,z,0))y_1(\tau,z)^2+12\dfrac{\p F_2}{\p
x}(\tau,x(\tau,z,0))y_2(\tau,z)\vspace{0.3cm}\\
&+12\dfrac{\p^2 F_1}{\p x^2}(\tau,x(\tau,z,0))y_1(\tau,z)\odot
y_2(\tau,z)+4\dfrac{\p^3 F_1}{\p
x^3}(\tau,x(\tau,z,0))y_1(\tau,z)^3\vspace{0.3cm}\\
&+4\dfrac{\p F_1}{\p x}(\tau,x(\tau,z,0))y_3(\tau,z)+3\dfrac{\p^2
F_0}{\p x^2}(\tau,x(\tau,z,0))y_2(\tau,z)^2 \vspace{0.3cm}\\
&+4\dfrac{\p^2 F_0}{\p x^2}(\tau,x(\tau,z,0))y_1(\tau,z)\odot y_3(\tau,z)\vspace{0.3cm}\\
&+6\dfrac{\p^3 F_0}{\p x^3}(\tau,x(\tau,z,0))y_1(\tau,z)^2\odot
y_2(\tau,z)+\dfrac{\p^4 F_0}{\p
x^4}(\tau,x(\tau,z,0))y_1(\tau,z)^4\Bigg]\mathrm{d}\tau\vspace{0.3cm}\\
y_5(t,z)=&Y(t,z)\int_0^t
Y(\tau,z)^{-1}\Bigg[120F_5(\tau,x(\tau,z,0))+120\dfrac{\p F_4}{\p
x}(\tau,x(\tau,x,0))y_1(\tau,z) \vspace{0.3cm}\\
&+60\dfrac{\p^2 F_3}{\p
x^2}(\tau,x(\tau,z,0))y_1(\tau,z)^2+60\dfrac{\p F_3}{\p
x}(\tau,x(\tau,z,0))y_2(\tau,z)\vspace{0.3cm}\\
&+60\dfrac{\p^2 F_2}{\p x^2}(\tau,x(\tau,z,0))y_1(\tau,z)\odot
y_2(\tau,z)+20\dfrac{\p^3 F_2}{\p
x^3}(\tau,x(\tau,z,0))y_1(\tau,z)^3\vspace{0.3cm}\\
&+20\dfrac{\p F_2}{\p x}(\tau,x(\tau,z,0))y_3(\tau,z)+20\dfrac{\p^2
F_1}{\p x^2}(\tau,x(\tau,z,0))y_1(\tau,z)\odot y_3(\tau,z)\vspace{0.3cm}\\
&+15\dfrac{\p^2 F_1}{\p
x^2}(\tau,x(\tau,z,0))y_2(\tau,z)^2+30\dfrac{\p^3 F_1}{\p
x^3}(\tau,x(\tau,z,0))y_1(\tau,z)^2\odot y_2(\tau,z) \vspace{0.3cm}\\
&+5\dfrac{\p^4 F_1}{\p
x^4}(\tau,x(\tau,z,0))y_1(\tau,z)^4+5\dfrac{\p F_1}{\p
x}(\tau,x(\tau,z,0))y_4(\tau,z)\vspace{0.3cm}\\
&+10\dfrac{\p^2 F_0}{\p x^2}(\tau,x(\tau,z,0))y_1(\tau,z)\odot y_3(\tau,z)\vspace{0.3cm}\\
&+5\dfrac{\p^2 F_0}{\p x^2}(\tau,x(\tau,z,0))y_1(\tau,z)\odot y_4(\tau,z)\vspace{0.3cm}\\
&+15\dfrac{\p^3 F_0}{\p x^3}(\tau,x(\tau,z,0))y_1(\tau,z)\odot y_2(\tau,z)^2\vspace{0.3cm}\\
&+10\dfrac{\p^3 F_0}{\p x^3}(\tau,x(\tau,z,0))y_1(\tau,z)^2\odot y_3(\tau,z)\vspace{0.3cm}\\
&+10\dfrac{\p^4 F_0}{\p x^4}(\tau,x(\tau,z,0))y_1(\tau,z)^3\odot
y_2(\tau,z)+\dfrac{\p^5 F_0}{\p
x^5}(\tau,x(\tau,z,0))y_1(\tau,z)^5\Bigg]\mathrm{d}\tau.
\end{align*}

\section*{Appendix B: Basic results on the Brouwer degree}

In this appendix, following the
Browder's paper \cite{B}, we present the existence and uniqueness result from
the degree theory in finite dimensional spaces.

\begin{theorem}\label{ApAt1}
Let $X=\R^n=Y$ for a given positive integer $n$. For bounded open
subsets $V$ of $X$, consider continuous mappings
$f:\cl(V)\rightarrow Y$, and points $y_0$ in $Y$ such that $y_0$
does not lie in $f(\partial V)$ (as usual $\partial V$ denotes the
boundary of $V$). Then to each such triple $(f,V,y_0)$, there
corresponds an integer $d(f,V,y_0)$ having the following three
properties.
\begin{itemize}
\item[(i)] If $d(f,V,y_0)\neq 0$, then $y_0\in f(V)$. If $f_0$
is the identity map of $X$ onto $Y$, then for every bounded open set
$V$ and $y_0\in V$, we have
\[
d\left(f_0\big|_V,V,y_0\right)=\pm 1.
\]

\item[(ii)] $($Additivity$)$ If $f:\cl(V)\rightarrow Y$ is
a continuous map with $V$ a bounded open set in $X$, and $V_1$ and
$V_2$ are a pair of disjoint open subsets of $V$ such that
\[
y_0\notin f(\cl(V)\backslash(V_1\cup V_2)),
\]
then,
\[
d\left(f_0,V,y_0\right)=d\left(f_0,V_1,y_0\right)+
d\left(f_0,V_1,y_0\right).
\]

\item[(iii)] $($Invariance under homotopy$)$ Let $V$ be a
bounded open set in $X$, and consider a continuous homotopy
$\{f_t:0\leq t\leq 1\}$ of maps of $\cl(V)$ in to $Y$. Let
$\{y_t:0\leq t\leq 1\}$ be a continuous curve in $Y$ such that
$y_t\notin f_t(\partial V)$ for any $t\in[0,1]$. Then $d(f_t,V,y_t)$
is constant in $t$ on $[0,1]$.
\end{itemize}
Moreover the degree function $d(f,V,y_0)$ is uniquely determined by the three
above conditions.
\end{theorem}

\section*{Acknowledgements}

The first author is partially supported by CNPq 248501/2013-5. The
second author is partially supported by a FEDER-MINECO grant
MTM2016-77278-P, a MINECO grant MTM2013-40998-P, and an AGAUR grant
number 2014SGR-568. The third author is supported by a FAPESP grant 2016/11471-2.

\end{document}